\documentclass[reqno,12pt]{amsart}
\usepackage{amsfonts}
\numberwithin{equation}{section}
\usepackage{tikz-cd}
\usetikzlibrary{bending,arrows.meta}

\usepackage{indentfirst}
\usepackage{color}
\definecolor{revisionpurple}{RGB}{128,64,160}
\definecolor{sodA}{RGB}{39,99,140}
\definecolor{sodB}{RGB}{169,99,6}
\definecolor{sodM}{RGB}{104,114,125}
\usepackage{amssymb}
\usepackage{mathrsfs}
\usepackage{xy}
\xyoption{all}
\def\Ext{\mbox{\rm Ext}\,} \def\Hom{\mbox{\rm Hom}} \def\dim{\mbox{\rm dim}} 
    \def\mod{\mbox{\rm {mod}}}
   \def\im{\mbox{\rm Im}\,} 
\def\End{\mbox{\rm End}\,}

\theoremstyle{plain}
\newtheorem{theorem}{\bf Theorem}[section]
\newtheorem{lemma}[theorem]{\bf Lemma}
\newtheorem{corollary}[theorem]{\bf Corollary}
\newtheorem{proposition}[theorem]{\bf Proposition}

\theoremstyle{definition}
\newtheorem{definition}[theorem]{\bf Definition}
\newtheorem{remark}[theorem]{\bf Remark}
\newtheorem{example}[theorem]{\bf Example}

\newcommand{\bt}{\begin{theorem}}
\newcommand{\et}{\end{theorem}}
\newcommand{\bl}{\begin{lemma}}
\newcommand{\el}{\end{lemma}}
\newcommand{\bd}{\begin{definition}}
\newcommand{\ed}{\end{definition}}
\newcommand{\bc}{\begin{corollary}}
\newcommand{\ec}{\end{corollary}}
\newcommand{\bp}{\begin{proof}}
\newcommand{\ep}{\end{proof}}
\newcommand{\bx}{\begin{example}}
\newcommand{\ex}{\end{example}}
\newcommand{\br}{\begin{remark}}
\newcommand{\er}{\end{remark}}
\newcommand{\be}{\begin{equation}}
\newcommand{\ee}{\end{equation}}
\newcommand{\ba}{\begin{align}}
\newcommand{\ea}{\end{align}}
\newcommand{\bn}{\begin{enumerate}}
\newcommand{\en}{\end{enumerate}}
\newcommand{\bcs}{\begin{cases}}
\newcommand{\ecs}{\end{cases}}

\makeatletter
\renewcommand{\section}{\@startsection{section}{1}{0mm}
  {-\baselineskip}{0.5\baselineskip}{\bf\leftline}}
\makeatother

\begin{document}
\title[From \(\tau\)-Tilting to \(\tau_2\)-Tilting via Splitting Tilting Triples]{From \(\tau\)-Tilting to \(\tau_2\)-Tilting via Splitting Tilting Triples}
\author{Zhongyuan Wang, Haicheng Zhang, Tiwei Zhao}
\address[Z. Wang, H. Zhang]{Ministry of Education Key Laboratory of NSLSCS, School of Mathematical Sciences, Nanjing Normal University, Nanjing 210023, P. R. China}

\email{2567273369@qq.com (Z. Wang), zhanghc@njnu.edu.cn (H. Zhang)}

\address[T. Zhao]{School of Artificial Intelligence, Jianghan University,
Wuhan 430056, P. R. China}
\email{tiweizhao@jhun.edu.cn}


\begin{abstract}
Let $(A,T,B)$ be a splitting tilting triple. We provide a construction of $\tau_2$-tilting $B$-modules from $\tau$-tilting $A$-modules.
On the other hand, if $A$ is hereditary, by considering the mutation of a certain two-term silting complex of $A$
and its image under the derived equivalence $\mathbf{R}\Hom_A(T,-):D^b(A)\to D^b(B)$, we get a three-term silting complex ${\bf W}$ of $B$.
Under an explicit non-vanishing condition on the cones of morphisms between the two summands, we prove that ${\rm H}^0\boldsymbol{W}$ is a basic $2$-tilting $B$-module and that it coincides with the module obtained from the first construction.
As an application, we generalize the construction of Peng \cite{PENG} from tilting modules to $\tau$-tilting modules.
We also determine the endomorphism algebra of the constructed module as an
explicit triangular matrix algebra, which is derived equivalent to $A$ and $B$ whenever the constructed module is
$2$-tilting, and  give
a full semi-orthogonal decomposition of $K^b(\operatorname{proj}A)$.
\end{abstract}

\maketitle

\section{Introduction}

Tilting theory is one of the most fundamental tools in representation theory of algebras, establishing a bridge between module categories and derived categories of two algebras. The classical tilting modules based on the BGP-reflection functors in \cite{Ausl,BGP,BB} are formalized by
Happel--Ringel \cite{Happel}, which have been generalized to tilting modules of finite projective dimension in \cite{MIY,Happel2}. Given a tilting module $T$ over an algebra,
the well-known Brenner--Butler theorem establishes a correspondence between torsion pairs in the category of $A$-modules and those in the module category of the endomorphism algebra $B:=\operatorname{End}_A T$,
thereby providing a powerful framework for transferring homological information between the two algebras (cf. \cite{ASS}).

In recent years, the theory of $\tau$-tilting modules, introduced by Adachi--Iyama--Reiten
\cite{ADA}, has emerged as a significant generalization of classical tilting theory, giving
a completion of tilting theory from the point of view of the mutations arising
from the theory of cluster algebras (cf. \cite{Fomin1,Fomin2}). Subsequently, the higher analogue of $\tau$-tilting theory, known as $\tau_n$-tilting theory, has been developed by Muchtadi-Alamsyah and Palu \cite{MUC} and further explored by Mart\'{\i}nez and Mendoza \cite{MAR}. These higher $\tau_n$-tilting modules are intimately related with $(n+1)$-term silting complexes in the bounded homotopy category of projective modules, establishing a rich interplay between module categories and derived categories.

Within the classical tilting theory, a distinguished role is played by splitting tilting modules, for which the induced torsion pairs are splitting. Such modules exhibit particularly nice homological properties. For instance, Hoshino \cite{HOS} and Chen--Zhang \cite{CHE} characterized splitting tilting triples in terms of the Auslander--Reiten translation and the injective dimensions of torsion-free modules. These splitting tilting modules provide a fertile ground for constructing new tilting-like objects.

A central theme in the development of tilting theory is the construction and mutation of tilting modules and their higher analogues. Peng \cite{PENG} investigated the behavior of tilting modules under the Brenner--Butler correspondence, and showed that for a hereditary algebra $A$, a certain construction involving classical tilting $A$-modules yields $2$-tilting modules over the corresponding tilted algebra $B$. On the other hand, the mutation theory of silting complexes, developed by Aihara--Iyama \cite{AIHA}, provides a systematic method for producing new silting complexes from existing ones.

The present paper aims to unify and extend these strands of research by establishing a systematic connection between
$\tau_n$-tilting modules and silting complexes in the context of splitting tilting triples. Our main results are threefold. First, we prove that given a splitting tilting triple $(A,T,B)$, the Brenner--Butler correspondence can be refined to construct a $\tau_{2}$-tilting $B$-module from a
$\tau$-tilting $A$-module satisfying some orthogonality conditions (see Theorem \ref{3.3}). Under explicit trace and orthogonality hypotheses, this construction extends the main result of Peng \cite{PENG} from classical tilting modules over hereditary algebras to $\tau$-tilting modules over finite-dimensional algebras admitting a splitting tilting triple.
Second, we investigate the derived category counterpart of this construction. We show that for a hereditary algebra
$A$, the right mutation of any two-term silting complex with cohomologies lying in the corresponding subcategories determined by $T$ induces
a three-term silting complex $\boldsymbol{W}$ of $B$ (see Theorem \ref{main1}).
Under the additional cone non-vanishing condition in Proposition \ref{mac}, we prove that the cohomology ${\rm H}^0\boldsymbol{W}$ is a basic $2$-tilting $B$-module which is isomorphic to the one constructed by Peng \cite{PENG}. Hence, this provides
an interpretation of the construction in \cite{PENG} from the perspective of mutation constructions of silting complexes.
Third, we characterize explicitly the endomorphism algebra of the constructed module, and give a full semi-orthogonal decomposition of
$K^b(\operatorname{proj}A)$.

The paper is organized as follows. In Section 2, we recall some necessary preliminaries in tilting theories.
Section 3 is devoted to giving the construction of $\tau_2$-tilting modules from splitting tilting triples. In Section 4, we establish the derived category counterpart of the construction by applying mutations of silting complexes, and explain the relationship between these two constructions.
In Section 5, we describe the endomorphism algebra of the constructed module
and, when it is $2$-tilting, derive the resulting derived equivalences and
semi-orthogonal decompositions.

Throughout the paper, for any object $X$ in an additive category, we denote by ${\rm add} X$ the subcategory consisting of direct summands of finite direct sums of $X$;
we suppose $A$ is a finite-dimensional $k$-algebra over a field $k$, all modules are finitely generated modules, all subcategories are full and closed under taking direct summands. Denote by $\operatorname{mod} A$ the category of finitely generated right $A$-modules, by $\operatorname{proj}A$ the subcategory of $\operatorname{mod} A$ consisting of all projective $A$-modules, by $D^b(A)$ the bounded derived category of $\operatorname{mod}A$, and by $K^b(\operatorname{proj}A)$ the bounded homotopy category of projective $A$-modules. Let $\tau$ be the Auslander--Reiten translation of $\operatorname{mod} A$. For any module $M$, denote by $M^{\oplus n}$ the direct sum of $n$ copies of $M$ for any positive integer $n$; denote by $|M|$ the number of pairwise non-isomorphic indecomposable direct summands of $M$; denote by $\operatorname{pd} M$ and $\operatorname{id} M$ the projective dimension and injective dimension of $M$, respectively; denote by ${\rm Gen} M$ the subcategory of modules that are generated by $M$, i.e. the modules $N$ admitting an epimorphism $M'\to N$ with $M'\in{\rm add} M$. We denote by $\mathbb{D}$ the $k$-linear duality. For a finite set $S$, we denote by $\# S$ its cardinality.

\section{Preliminaries}

In this section, we recall some necessary preliminaries in tilting theories.

\begin{definition}\label{tilting} (\cite{Happel})
An \( A \)-module \( T \) is called a (classical) {\em tilting module} if the following conditions are satisfied:
\begin{enumerate}
\item[(T1)] \( \operatorname{pd} T \leq 1 \);
\item[(T2)] \( \operatorname{Ext}_A^1(T, T) = 0 \);
\item[(T3)] There exists a short exact sequence
$$\xymatrix{0 \ar[r]& A \ar[r]& T^{0} \ar[r]&T^{1} \ar[r]& 0}$$
with \( T^{0} , T^{1}  \in  {\rm add} T \).
\end{enumerate}
\end{definition}

\begin{definition}\label{ntilting}(\cite{MIY})
Let $n$ be a positive integer. An \( A \)-module \( T \) is called an {\em $n$-tilting module} if the following conditions are satisfied:
\begin{enumerate}
\item[(T1)] \( \operatorname{pd} T \leq n \),
\item[(T2)] \( \operatorname{Ext}_A^i(T, T) = 0 \) for $i=1 , 2 ,\dots , n$.
\item[(T3)] There exists an exact sequence
$$\xymatrix{0 \ar[r]& A \ar[r]& T^{0}\ar[r]& T^{1} \ar[r]& \cdots \ar[r]& T^{n} \ar[r]& 0}$$
with $T^{0} , \cdots , T^{n}  \in {\rm add} T $.
\end{enumerate}
\end{definition}
Clearly, $1$-tilting modules are exactly the tilting modules in Definition \ref{tilting}.

\begin{definition}\label{tp} (\cite{ASS})
A pair \((\mathcal{T}, \mathcal{F})\) of full subcategories of \(\mathrm{mod}\, A\) is called a {\em torsion pair} if the following conditions are satisfied:
\begin{enumerate}
\item[(1)] $\operatorname{Hom}_A(\mathcal{T},\mathcal{F})=0$.
\item[(2)] For any $X\in{\rm mod} A$, there exists a short exact sequence
\begin{equation}\label{tfenjie}
\xymatrix{0 \ar[r]& {\bf t} X \ar[r]& X \ar[r]& {\bf f} X \ar[r]& 0}\end{equation}
such that ${\bf t} X \in \mathcal{T}$ and ${\bf f} X \in \mathcal{F}$.
\end{enumerate}
A torsion pair \((\mathcal{T}, \mathcal{F})\) is {\em splitting} if the exact sequence \eqref{tfenjie} splits for any $X\in\mod A$. Refer to \cite[VI Proposition 1.7]{ASS} for equivalent characterizations of splitting torsion pairs.
\end{definition}
\begin{remark}
Let \((\mathcal{T}, \mathcal{F})\) be a torsion pair. Then we have the following:

$(1)$~The subcategories $\mathcal{T}$ and $\mathcal{F}$ are extension-closed, and they determine each other, namely,
\begin{flalign*}
&\mathcal{T}=\{M\in\mod A~|~\Hom_A(M,\mathcal{F})=0\}=:{^\bot\mathcal{F}}\\
&\mathcal{F}=\{N\in\mod A~|~\Hom_A(\mathcal{T},N)=0\}=:\mathcal{T}^\bot.
\end{flalign*}
Moreover, $\mathcal{T}$ is closed under quotients, and $\mathcal{F}$ is closed under subobjects (cf. \cite[Remark 2.4]{BAU}).

$(2)$~The exact sequence \eqref{tfenjie} is unique in the sense that, if
\begin{equation}\label{tfenjie2}
\xymatrix{0 \ar[r]& X' \ar[r]& X \ar[r]& X'' \ar[r]& 0}\end{equation}
is exact with $X'\in\mathcal{T}$, $X''\in \mathcal{F}$, then the sequences \eqref{tfenjie} and \eqref{tfenjie2} are equivalent (cf. \cite[VI Proposition 1.5]{ASS}). In particular, ${\bf t} X\cong X'$ and ${\bf f} X\cong X''$.

$(3)$~The exact sequence \eqref{tfenjie} is functorial. That is, for any morphism $f:X\rightarrow X'$, fix the decomposition sequences \eqref{tfenjie} of $X$ and $ X'$, then we can complete the following diagram uniquely (cf. \cite[Remark 2.4]{BAU}):
\begin{equation}\label{2.3}
\xymatrix{0 \ar[r]& {\bf t} X  \ar[r]\ar@{-->}[d]_-{{\bf t}f}& X \ar[r]\ar[d]^-f&  {\bf f} X  \ar[r]\ar@{-->}[d]^-{{\bf f}f}& 0\\
0 \ar[r]& {\bf t} X'  \ar[r]& X' \ar[r]& {\bf f} X' \ar[r]& 0.}
\end{equation}
$(4)$~There are natural isomorphisms (cf. \cite[Remark 2.4]{BAU})
\begin{flalign}
&\Hom_A(M,N)\cong\Hom_{A}({\bf f}M, N)~\text{for~any}~M\in\mod A, N\in \mathcal{F};\\
&\Hom_A(M,N)\cong\Hom_{A}(M, {\bf t}N)~\text{for~any}~M\in\mathcal{T}, N\in \mod A\label{thomtg}.
\end{flalign}
\end{remark}
In what follows, we fix a decomposition sequence \eqref{tfenjie} of $X$ for any $X\in\mod A$. Then the diagram \eqref{2.3} defines the functors ${\bf t}:\mod A\rightarrow \mathcal{T}$ and ${\bf f}:\mod A\rightarrow \mathcal{F}$.

Given a tilting $A$-module $T$, define
\(\mathcal{T}(T)=\{M \in {\rm mod} A\,|\,\Ext_A^{1}(T,M)=0\}\) and
\(\mathcal{F}(T)=\{M \in {\rm mod} A\,|\,\Hom_A(T,M)=0\}\).
By \cite{ASS}, $(\mathcal{T}(T), \mathcal{F}(T))$ is a torsion pair in $\mod A$. Let $B=\operatorname{End}_A T$. By \cite[VI Corollary 3.6]{ASS}, the tilting $A$-module $T$ also induces a torsion pair $(\mathcal{X}(T_A), \mathcal{Y}(T_A))$ in $\mod B$, which is given by
\begin{flalign*}
&\mathcal{X}(T)=\{X~|~\Hom_B(X,\mathbb{D}T)=0\}=\{X~|~X\otimes_B T=0\},\\
&\mathcal{Y}(T)=\{Y~|~\Ext^1_B(Y,\mathbb{D}T)=0\}=\{Y~|~{\rm Tor}_1^B(Y,T)=0\}.
\end{flalign*}
A triple \( (A, T, B) \) is called a {\em tilting triple} if $T$ is a tilting $A$-module and \(B=\operatorname{End}_A T\).

\begin{theorem}\label{BB} \textup{(\cite[Theorem 3.8]{ASS})}
Let \( (A, T, B) \) be a tilting triple. Then we have the following:
\begin{enumerate}
\item[(1)] The functors \( \operatorname{Hom}_A(T, -) \) and \( - \otimes_B T \) induce quasi-inverse equivalences between \( \mathcal{T}(T) \) and \( \mathcal{Y}(T) \).
\item[(2)] The functors \( \operatorname{Ext}_A^1(T, -) \) and \( \operatorname{Tor}_1^B(-, T) \) induce quasi-inverse equivalences between \( \mathcal{F}(T) \) and \( \mathcal{X}(T) \).
\end{enumerate}
\end{theorem}
It is well-known that for every tilting triple \( (A, T, B) \), we have a derived equivalence (cf. \cite[Theorem~2.1]{CLI}, \cite[Theorem 6.4]{RIC})
\begin{equation}\label{RHom}
\mathbf{R}\Hom_A(T, -): D^b(A) \longrightarrow D^b(B).\end{equation}

According to \cite[VI Definition 5.1]{ASS}, we call a tilting module $T$ {\em splitting} if the torsion pair $(\mathcal{X}(T),\mathcal{Y}(T))$ in $\operatorname{mod} B$ is splitting.
In this case, the tilting triple \( (A, T, B) \) is called a {\em splitting tilting triple}.
We have the following characterizations of splitting tilting modules (cf. \cite[Lemma 2.10]{CHE}, \cite{HOS}).
\begin{lemma}\label{splitting}
Let \( (A, T, B) \) be a tilting triple. Then the following are equivalent.
\begin{enumerate}
\item[(1)] \( T_A \) is a splitting tilting module.
\item[(2)] \( \operatorname{Hom}_A(T, \tau_A M) \cong \tau_B \operatorname{Hom}_A(T, M) \) for any \( M \in \mathcal{T}(T_A) \).
\item[(3)] \( \operatorname{Ext}_A^1(T, \tau_A^{-1} N) \cong \tau_B^{-1} \operatorname{Ext}_A^1(T, N) \) for any \( N \in \mathcal{F}(T_A) \).
\item[(4)] \( \operatorname{inj.dim}\, N \leq 1 \) for any \( N \in \mathcal{F}(T_A) \).
\end{enumerate}
\end{lemma}
For later use, we record the following extension of
\cite[Lemmas 1.3--1.6]{PENG}.
\begin{lemma}\label{xt}
Let \( (A, T, B) \) be a tilting triple.
\begin{enumerate}
\item[(1)] For any $M , N \in \mathcal{T} (T)$, we have
\[\operatorname{Ext}_B^i(\operatorname{Hom}_A ( T , M ) , \operatorname{Hom}_A ( T , N )) \cong \operatorname{Ext}_A^i (M ,N),~i=0,1.\]
\item[(2)] For any $M , N \in \mathcal{F} (T)$, we have
\[\operatorname{Ext}_B^i(\operatorname{Ext}_A^1 ( T , M ), \operatorname{Ext}_A^1 ( T , N )) \cong \operatorname{Ext}_A^i (M ,N),~i=0,1.\]
\item[(3)] For any $M \in \mathcal{T} (T)$ and $N \in \mathcal{F} (T)$, we have
\[\operatorname{Ext}_B^i(\operatorname{Hom}_A ( T , M ), \operatorname{Ext}_A^1 ( T, N )) \cong \operatorname{Ext}_A^{i+1} (M ,N),~i \ge 0.\]
\item[(4)] For any $M \in \mathcal{F} (T)$ and $N \in \mathcal{T} (T)$, we have
\[\operatorname{Ext}_B^i(\operatorname{Ext}_A^1 ( T , M ) , \operatorname{Hom}_A ( T , N )) \cong \operatorname{Ext}_A^{i-1} (M ,N),~i \ge 1.\]
\end{enumerate}
\end{lemma}
\begin{proof}
For any $L\in\mathcal{T}(T)$, the equivalence \eqref{RHom} gives
\(\mathbf{R}\Hom_A(T,L)\cong\Hom_A(T,L)\). For any
$K\in\mathcal{F}(T)$, it gives
\(\mathbf{R}\Hom_A(T,K)\cong\Ext_A^1(T,K)[-1]\).
The four asserted isomorphisms now follow by computing morphisms between these objects and their shifts in $D^b(B)$. This argument also shows that the assertions hold for an arbitrary tilting triple, without a hereditary assumption on $A$.
\end{proof}

\begin{definition}\label{tau} \textup{(\cite[Definition 0.1]{ADA})}
Let $M$ be an $A$-module.
\begin{enumerate}
\item[(a)] The module $M$ is called a {\em \(\tau\)-rigid} module if
    $\operatorname{Hom}_{A}(M, \tau M) = 0.$
\item[(b)] The module $M$ is called a {\em \(\tau\)-tilting} module if \( M \) is \(\tau\)-rigid and
    $|M| = |A|.$
\item[(c)] The module $M$ is called a {\em support \(\tau\)-tilting} module if there exists an idempotent \( e \) of \( A \) such that \( M \) is a \(\tau\)-tilting module over \( A / \langle e \rangle \).
\end{enumerate}
\end{definition}

According to \cite{IYA}, for any positive integer $n$, define $\tau_n:=\tau_{A,n}$ by $\tau_n M=\tau (  \Omega^{n-1} M )$ for any $A$-module $M$, where $\Omega^{0} M:=M$ and $\Omega^{n-1} M$ is the $(n-1)$-th syzygy of $M$. Clearly, we have $\tau_1=\tau$.

\begin{definition}\label{tau-n}\textup{(\cite[Definition 2.2]{MUC})}
Let $n$ be a positive integer and $M\in\mod A$.
\begin{enumerate}
\item[(a)] The module $M$ is called a {\em strongly \( \tau_n \)-rigid} module if $\operatorname{Hom}_A (M, \tau_n M) = 0$ and $\operatorname{Ext}_A^i(M, M) = 0$ for any $1\leq i<n.$
\item[(b)] The module $M$ is called a {\em \( \tau_n \)-{tilting}} module if \( M \) is a strongly \( \tau_n \)-rigid module and $ |M| = |A|$.
\item[(c)] The module $M$ is called a {\em support \( \tau_n \)-tilting} module if there exists an idempotent \( e \) of \( A \) such that \( M \) is a \( \tau_n \)-tilting module over \( A/\langle e \rangle \).
\end{enumerate}
\end{definition}

\begin{remark}\label{n-1,n}
As pointed out in \cite[Remark 2.1]{MUC} , the functor $\tau_{n}$ in \cite{IYA} corresponds to the functor $\tau_{n+1}$ in \cite{MUC}. Here we adopt the notation from \cite{IYA}.
After reindexing, the additional condition $\Ext_A^n(M,M)=0$ in \cite[Definition 2.2]{MUC} is automatic from $\Hom_A(M,\tau_{A,n}M)=0$ by Lemma \ref{tau-gen}, since $M\in\operatorname{Gen} M$. Thus Definition \ref{tau-n} agrees with the reindexed definition in \cite{MUC}.
\end{remark}

\begin{lemma}\label{tau-gen} \textup{(\cite[Corollary 3.5(c)]{MAR})}
Let $n$ be a positive integer. For any $A$-modules $X$ and $Y$, we have
$\operatorname{Hom}_{A}(X, \tau_n Y) = 0~\text{if and only if}~\operatorname{Ext}_{A}^{n}(Y, \operatorname{Gen} X) = 0.$
\end{lemma}

\section{The $\tau_2$-tilting $B$-modules arising from $\tau$-tilting $A$-modules}
In this section, let \( (A, T, B) \) be a tilting triple. Given a certain $\tau$-tilting $A$-module, we study how to get a $\tau_2$-tilting $B$-module.
\begin{lemma}\label{3.1}
Let \( (A, T, B) \) be a splitting tilting triple. Let $M,M' \in \mathcal{T}(T_A)$ and $N \in \mathcal{F}(T_A)$. Then the following statements hold:
\begin{enumerate}
\item[(1)] \( \operatorname{Hom}_B ( \operatorname{Hom}_A(T,M) , \tau_{B}\operatorname{Hom}_A(T,M') ) \cong \operatorname{Hom}_A(M,\tau_A  M') .\)
\item[(2)] \(\operatorname{Hom}_A(M, \tau_{A,2} M') = 0\) if and only if
\(\operatorname{Hom}_B(\operatorname{Hom}_A(T, M),\tau_{B,2} \operatorname{Hom}_A(T, M')) = 0.\)
\item[(3)]\(\operatorname{Hom}_B(\operatorname{Ext}_A^1(T, N),\tau_{B,2} \operatorname{Ext}_A^1(T, N)) = 0.\)
\item[(4)]\(\operatorname{Hom}_B(\operatorname{Ext}_A^1(T, N), \tau_{B,2} \operatorname{Hom}_A(T, M)) = 0.\)
\item[(5)] \( \operatorname{Hom}_A(M,\tau_A N) = 0 \) if and only if \(\operatorname{Hom}_B(\operatorname{Hom}_A(T, M), \tau_{B,2} \operatorname{Ext}_A^1(T, N)) = 0.\)
\end{enumerate}
\end{lemma}
\begin{proof}
$(1)$~By Lemma \ref{splitting}, \eqref{thomtg} and Lemma \ref{xt}, we have the following isomorphisms
\begin{equation*}
\begin{split}
 \operatorname{Hom}_B ( \operatorname{Hom}_A(T,M) , \tau_{B}\operatorname{Hom}_A(T,M') )
& \cong \operatorname{Hom}_B ( \operatorname{Hom}_A(T,M) , \operatorname{Hom}_A(T,\tau_{A} M') ) \\
& \cong \operatorname{Hom}_B ( \operatorname{Hom}_A(T,M) , \operatorname{Hom}_A(T, {\bf t}(\tau_{A} M')) ) \\
& \cong \operatorname{Hom}_A(M, {\bf t}(\tau_{A} M')) \\
& \cong \operatorname{Hom}_A(M, \tau_{A} M').
\end{split}
\end{equation*}

$(2)$ By \cite[Theorem 2.5(c) and Corollary 2.6]{ASS}, there exists a
short exact sequence
\begin{equation}\label{star1}
0\longrightarrow M_1\longrightarrow T_0\longrightarrow M'
\longrightarrow0
\end{equation}
with $T_0\in\operatorname{add}T$ and
$M_1\in\mathcal{T}(T)$. Applying $\operatorname{Hom}_A(T,-)$ to
\eqref{star1} gives an exact sequence
\[
0\longrightarrow\operatorname{Hom}_A(T,M_1)
\longrightarrow\operatorname{Hom}_A(T,T_0)
\longrightarrow\operatorname{Hom}_A(T,M')
\longrightarrow0,
\]
where $\operatorname{Hom}_A(T,T_0)$ is projective over $B$. Thus, we obtain that there exists a projective
$B$-module $P$ such that
$\operatorname{Hom}_A(T,M_1)
\cong
\Omega_B\operatorname{Hom}_A(T,M')\oplus P.$
Since the Auslander--Reiten translation vanishes on projective modules, we have
\[
\tau_{B,2}\operatorname{Hom}_A(T,M')
=
\tau_B\Omega_B\operatorname{Hom}_A(T,M')
\cong
\tau_B\operatorname{Hom}_A(T,M_1).
\]
Thus, by $(1)$ we get
\begin{align*}
\operatorname{Hom}_B
 \bigl(\operatorname{Hom}_A(T,M),
       \tau_{B,2}\operatorname{Hom}_A(T,M')\bigr)
&\cong
\operatorname{Hom}_B
 \bigl(\operatorname{Hom}_A(T,M),
       \tau_B\operatorname{Hom}_A(T,M_1)\bigr)\\
&\cong
\operatorname{Hom}_A(M,\tau_A M_1).
\end{align*}
For any $Z\in\operatorname{Gen}M$, we have
$Z\in\mathcal{T}(T)$, since
$\mathcal{T}(T)$ is closed under quotients. Applying $\operatorname{Hom}_A(-,Z)$ to
\eqref{star1} yields the exact sequence
\[
\operatorname{Ext}_A^1(T_0,Z)
\longrightarrow
\operatorname{Ext}_A^1(M_1,Z)
\longrightarrow
\operatorname{Ext}_A^2(M',Z)
\longrightarrow
\operatorname{Ext}_A^2(T_0,Z).
\]
Note that $\operatorname{Ext}_A^1(T_0,Z)=0$, since
$T_0\in\operatorname{add}T$ and $Z\in\mathcal{T}(T)$;
$\operatorname{Ext}_A^2(T_0,Z)=0$, since
$\operatorname{pd}_A T_0\leq1$. Thus,
$\operatorname{Ext}_A^1(M_1,Z)
\cong
\operatorname{Ext}_A^2(M',Z)$ for any $Z\in\operatorname{Gen}M.$
It follows that
\[
\operatorname{Ext}_A^1(M_1,\operatorname{Gen}M)=0
\quad\Longleftrightarrow\quad
\operatorname{Ext}_A^2(M',\operatorname{Gen}M)=0.
\]
Hence, applying Lemma \ref{tau-gen}, we obtain
\begin{align*}
\operatorname{Hom}_B
 \bigl(\operatorname{Hom}_A(T,M),
       \tau_{B,2}\operatorname{Hom}_A(T,M')\bigr)=0
&\quad\Longleftrightarrow
\operatorname{Hom}_A(M,\tau_A M_1)=0\\
&\quad\Longleftrightarrow
\operatorname{Ext}_A^1(M_1,\operatorname{Gen}M)=0\\
&\quad\Longleftrightarrow
\operatorname{Ext}_A^2(M',\operatorname{Gen}M)=0\\
&\quad\Longleftrightarrow
\operatorname{Hom}_A(M,\tau_{A,2}M')=0
\end{align*}
and finish the proof of $(2)$.

\par\medskip\noindent
$(3)$ By Lemma \ref{tau-gen}, it suffices to show that
\begin{equation}\label{xyds}
\operatorname{Ext}_B^2(\operatorname{Ext}_A^1(T, N), \operatorname{Gen}\operatorname{Ext}_A^1(T, N)) = 0.\end{equation}
Since $T$ is splitting, by Lemma \ref{splitting}, we have
$\operatorname{id}_A L\leq1$ for any $L\in\mathcal F(T)$.
For any $X\in\mathcal X(T)$, Theorem \ref{BB} gives
$L\in\mathcal F(T)$ such that
$X\cong\operatorname{Ext}_A^1(T,L)$. Thus, by \cite[VI, Exercise 20(c)]{ASS}, we have
\(\operatorname{id}_B X\leq\operatorname{id}_A L\leq1\). Since $\mathcal{X}(T)$ is closed under quotients, we conclude that
$\operatorname{Gen}\operatorname{Ext}_A^1(T, N) \subset \mathcal{X}(T)$, and then get the proof of \eqref{xyds}.

\par\medskip\noindent
$(4)$ Take any exact sequence
\begin{equation}\label{exsq34}
0\longrightarrow\widetilde M_1
\longrightarrow\widetilde T_0
\longrightarrow M
\longrightarrow0\end{equation}
with $\widetilde T_0\in\operatorname{add}T$ and
$\widetilde M_1\in\mathcal{T}(T)$. Applying
$\operatorname{Hom}_A(T,-)$ to \eqref{exsq34}, we obtain
\[
\operatorname{Hom}_A(T,\widetilde M_1)
\cong
\Omega_B\operatorname{Hom}_A(T,M)\oplus\widetilde P
\] for some projective $B$-module
$\widetilde P$.
Thus, using Lemma \ref{splitting} and \eqref{thomtg}, we get
\begin{align*}
\tau_{B,2}\operatorname{Hom}_A(T,M)
\cong
\tau_B\operatorname{Hom}_A(T,\widetilde M_1)
\cong
\operatorname{Hom}_A(T,\tau_A\widetilde M_1)
\cong
\operatorname{Hom}_A
 \bigl(T,{\bf t}(\tau_A\widetilde M_1)\bigr)
\in\mathcal{Y}(T).
\end{align*}
On the other hand,
$\operatorname{Ext}_A^1(T,N)\in\mathcal{X}(T)$. Since
$\operatorname{Hom}_B(\mathcal{X}(T),\mathcal{Y}(T))=0$, we conclude that
\[
\operatorname{Hom}_B
 \bigl(\operatorname{Ext}_A^1(T,N),
       \tau_{B,2}\operatorname{Hom}_A(T,M)\bigr)=0
\]
and finish the proof of $(4)$.

\par\medskip\noindent
$(5)$ Set $U:=\operatorname{Ext}_A^1(T,N)$ and
$V:=\operatorname{Hom}_A(T,M)$. Applying Lemma \ref{tau-gen}, we get
\[
\operatorname{Hom}_A(M,\tau_A N)=0
\quad\Longleftrightarrow\quad
\operatorname{Ext}_A^1(N,\operatorname{Gen}M)=0
\]
and
\[
\operatorname{Hom}_B(V,\tau_{B,2}U)=0
\quad\Longleftrightarrow\quad
\operatorname{Ext}_B^2(U,\operatorname{Gen}V)=0.
\]
Now it  suffices to prove
\[
\operatorname{Ext}_A^1(N,\operatorname{Gen}M)=0
\quad\Longleftrightarrow\quad
\operatorname{Ext}_B^2(U,\operatorname{Gen}V)=0.
\]

For the only-if direction, assume that
$\operatorname{Ext}_A^1(N,E)=0$ for any
$E\in\operatorname{Gen}M$, and take any
$W\in\operatorname{Gen}V$. Since the torsion pair
$(\mathcal{X}(T),\mathcal{Y}(T))$ is splitting, there are
$W_X\in\mathcal{X}(T)$ and $W_Y\in\mathcal{Y}(T)$ such that
$W\cong W_X\oplus W_Y$. Since $\operatorname{id}_B X\leq1$ for any $X\in\mathcal{X}(T)$,
we have $\operatorname{Ext}_B^2(U,W_X)=0$. By Theorem \ref{BB}, there exists $E\in\mathcal{T}(T)$ such that
$W_Y\cong\operatorname{Hom}_A(T,E)$ and $E\cong W_Y\otimes_B T$.
Since $W_Y$ is a direct summand of $W\in\operatorname{Gen}V$,
there is an epimorphism
$V^{\oplus r}\twoheadrightarrow W_Y$ for some positive integer $r$. Since $-\otimes_B T$ is right exact, the
Brenner--Butler counit isomorphisms give an epimorphism
\[\xymatrix{M^{\oplus r}
\cong
V^{\oplus r}\otimes_B T
\ar@{->>}[r]&
W_Y\otimes_B T
\cong E.}
\]
Thus, we get $E\in\operatorname{Gen}M$. By Lemma \ref{xt}(4), we obtain
\[
\operatorname{Ext}_B^2(U,W_Y)
\cong
\operatorname{Ext}_B^2
 \bigl(\operatorname{Ext}_A^1(T,N),
       \operatorname{Hom}_A(T,E)\bigr)
\cong
\operatorname{Ext}_A^1(N,E)
=0.
\]
Hence, we get
\[
\operatorname{Ext}_B^2(U,W)
\cong
\operatorname{Ext}_B^2(U,W_X)
\oplus
\operatorname{Ext}_B^2(U,W_Y)
=0
\]
and conclude that
$\operatorname{Ext}_B^2(U,\operatorname{Gen}V)=0$.

Conversely, assume that
$\operatorname{Ext}_B^2(U,W)=0$ for any
$W\in\operatorname{Gen}V$ and take any
$E\in\operatorname{Gen}M$. Since
$\operatorname{Gen}M\subseteq\mathcal{T}(T)$, we have
$E\in\mathcal{T}(T)$. Choose an epimorphism
$f:M^{\oplus r}\twoheadrightarrow E$ and set $K:=\operatorname{Ker}f$. Applying
$\operatorname{Hom}_A(T,-)$ to the exact sequence
\[
0\longrightarrow K\longrightarrow M^{\oplus r}
\xrightarrow{\,f\,}E\longrightarrow0,
\]
we get an exact sequence
\[
0\longrightarrow\operatorname{Hom}_A(T,K)
\longrightarrow V^{\oplus r}
\xrightarrow{\operatorname{Hom}_A(T,f)}
\operatorname{Hom}_A(T,E)
\longrightarrow\operatorname{Ext}_A^1(T,K)
\longrightarrow0,
\]
since $\operatorname{Ext}_A^1(T,M^{\oplus r})=0$. Set
$C:=\operatorname{Im}\operatorname{Hom}_A(T,f)$ and
$Q:=\operatorname{Ext}_A^1(T,K)$. Then $C\in\operatorname{Gen}V$ and there is a short exact sequence
\begin{equation}\label{cqzhl}
0\longrightarrow C
\longrightarrow\operatorname{Hom}_A(T,E)
\longrightarrow Q
\longrightarrow0.\end{equation}
Consider the torsion decomposition of $K$
\begin{equation}\label{xydzhl}
0\longrightarrow{\bf t}K
\longrightarrow K
\longrightarrow{\bf f}K
\longrightarrow0.
\end{equation}
Since $\mathbf{t}K\in\mathcal{T}(T)$,
$\mathbf{f}K\in\mathcal{F}(T)$, and $\operatorname{pd}_A T\leq1$,
by applying $\operatorname{Hom}_A(T,-)$ to \eqref{xydzhl}, we get
\[
Q=\operatorname{Ext}_A^1(T,K)
\cong
\operatorname{Ext}_A^1(T,{\bf f}K)
\in\mathcal{X}(T).
\]
Thus, $\operatorname{id}_B Q\leq1$ and then
$\operatorname{Ext}_B^2(U,Q)=0$. Note that $\operatorname{Ext}_B^2(U,C)=0$, since $C\in\operatorname{Gen}V$. Applying
$\operatorname{Hom}_B(U,-)$ to \eqref{cqzhl},
we obtain
\(\operatorname{Ext}_B^2
 \bigl(U,\operatorname{Hom}_A(T,E)\bigr)=0\).
By Lemma \ref{xt}(4), we get
\[
\operatorname{Ext}_A^1(N,E)
\cong
\operatorname{Ext}_B^2
 \bigl(\operatorname{Ext}_A^1(T,N),
       \operatorname{Hom}_A(T,E)\bigr)
=0
\]
and conclude that
$\operatorname{Ext}_A^1(N,\operatorname{Gen}M)=0$. Hence, we finish the proof of $(5)$.
\end{proof}

For any $M,N\in\mod A$, define
\(\operatorname{Tr}_MN :=
\sum\limits_{f\in{\rm Hom}_A(N,M)}\operatorname{Im}f\), and let
\begin{equation}\label{star2}
\xymatrix{0 \to \operatorname{Tr}_M N \ar[r]^-\lambda& M \ar[r]^-\pi& M / {\operatorname{Tr}_MN} \to 0}
\end{equation}
be the short exact sequence such that $\lambda$ and $\pi$ are the natural embedding and projection, respectively. We recall that a module is called {\em basic} if its indecomposable direct summands are pairwise non-isomorphic.

\begin{lemma}\label{3.2}
Let \(M\oplus N\) be a basic \(\tau\)-tilting $A$-module. Then the following statements hold:

$(1)$ \( (\operatorname{Gen} N, \mathcal{F}(N)) \) is a torsion pair, and the short exact sequence (\ref{star2}) can be viewed as a decomposition sequence of \( M \) with respect to this torsion pair.

$(2)$ The following equations hold:
\begin{enumerate}
\item[(i)] \( \operatorname{Ext}_A^1(N,\operatorname{Tr}_M N) = 0 \). {\rm(ii)}~\( \operatorname{Ext}_A^1(N, M / {\operatorname{Tr}_M N}) = 0 \). {\rm(iii)}~\( \operatorname{Ext}_A^1(M, M / {\operatorname{Tr}_M N}) = 0 \).
\item[(iv)] \( \operatorname{Ext}_A^1(M,\operatorname{Tr}_MN) = 0 \). {\rm(v)}~\( \operatorname{Hom}_A(\operatorname{Tr}_MN, M / {\operatorname{Tr}_M N}) = 0 \).
\item[(vi)] \( \operatorname{Hom}_A(N, M /{\operatorname{Tr}_MN}) = 0 \). {\rm(vii)}~\( \operatorname{Ext}_A^1(M / {\operatorname{Tr}_MN}, M /{\operatorname{Tr}_MN}) = 0 \).
\item[(viii)] If \( \operatorname{Hom}_A(M / {\operatorname{Tr}_MN}, \tau_{A,2}(M / \operatorname{Tr}_M N)) = 0 \), then
\(\operatorname{Ext}_A^1(\operatorname{Tr}_MN, M /{\operatorname{Tr}_M N}) = 0 \).
\end{enumerate}

$(3)$ Write $M=\bigoplus_{i=1}^nM_i$, where the $M_i$ are
pairwise non-isomorphic and indecomposable, and set
$X_i:=M_i/\operatorname{Tr}_{M_i}N$. Then $X_i=0$ if and only if $M_i\in\operatorname{Gen}N$. Every nonzero $X_i$ is indecomposable, and the nonzero $X_i$ are pairwise non-isomorphic. Consequently,
\[
\left|M/\operatorname{Tr}_MN\right|
=\#\{i\mid M_i\notin\operatorname{Gen}N\}.
\]
In particular, if no nonzero indecomposable direct summand of $M$ lies in $\operatorname{Gen}N$, then $M/\operatorname{Tr}_MN$ is basic and
$|M/\operatorname{Tr}_MN|=|M|$.
\end{lemma}
\begin{proof}
$(1)$ Since \( M \oplus N \) is \( \tau \)-tilting, \(\operatorname{Hom}_A(N, \tau N) = 0 \).
Thus, by Lemma \ref{tau-gen}, we have
\(\operatorname{Ext}_A^1(N, \operatorname{Gen} N) = 0.\)
By \cite[VI Lemma 1.9(a)]{ASS}, \( (\operatorname{Gen} N, \mathcal{F}(N)) \) is a torsion pair.

By definition, \( \operatorname{Tr}_M N \in \operatorname{Gen} N \). Applying \( \operatorname{Hom}_A(N, -) \) to (\ref{star2}), we obtain the exact sequence
\[0\to\operatorname{Hom}_A(N, \operatorname{Tr}_M N) \stackrel{a}{\rightarrow} \operatorname{Hom}_A(N, M) \to \operatorname{Hom}_A(N, M / \operatorname{Tr}_M N) \to \operatorname{Ext}_A^1(N, \operatorname{Tr}_M N).\]
By the definition of \( \operatorname{Tr}_M N \), we get the map $a$ is surjective, and then an isomorphism. Note that
\(\operatorname{Ext}_A^1(N, \operatorname{Tr}_M N) \subset \operatorname{Ext}_A^1(N, \operatorname{Gen} N) = 0.\)
Thus,
\(\operatorname{Hom}_A(N, M / \operatorname{Tr}_M N) = 0,\)
i.e.
\(M / \operatorname{Tr}_M N \in \mathcal{F}(N)\).
Hence, we finish the proof.

$(2)$ By definition, \( \operatorname{Tr}_M N \in \operatorname{Gen} N \) and \( M / \operatorname{Tr}_M N \in \operatorname{Gen} M \).
Since \( M \oplus N \) is \( \tau \)-tilting, we have \(\operatorname{Hom}_A( M \oplus N, \tau  M \oplus \tau  N) = 0 \). By Lemma \ref{tau-gen},
\(\operatorname{Ext}_A^1(M \oplus N, \operatorname{Gen} M \oplus \operatorname{Gen} N) = 0.\)
Thus, the equations (i), (ii), (iii) and (iv) hold.

Noting that \( N, \operatorname{Tr}_M N \in \operatorname{Gen} N \) and \( M / \operatorname{Tr}_M N \in \mathcal{F}(N) \), we obtain the equations (v) and (vi) by (1).
Applying \( \operatorname{Hom}(-, M / \operatorname{Tr}_M N) \) to (\ref{star2}), we get the exact sequence
\begin{flalign*}
&\operatorname{Hom}_A(\operatorname{Tr}_M N, M / \operatorname{Tr}_M N) \to \operatorname{Ext}_A^1(M / \operatorname{Tr}_M N, M / \operatorname{Tr}_M N) \to \operatorname{Ext}_A^1(M, M / \operatorname{Tr}_M N)\\& \to \operatorname{Ext}_A^1(\operatorname{Tr}_M N, M / \operatorname{Tr}_M N)
\to \operatorname{Ext}_A^2(M / \operatorname{Tr}_M N, M / \operatorname{Tr}_M N).
\end{flalign*}
By (v) and (iii), we know the first and third terms are zero. Thus the second term is zero, i.e. the equation (vii) holds.

If \( \operatorname{Hom}_A(M / \operatorname{Tr}_M N, \tau_{A,2}(M / \operatorname{Tr}_M N)) = 0 \), then by Lemma \ref{tau-gen}, we have
$$\operatorname{Ext}_A^2(M / \operatorname{Tr}_M N, M / \operatorname{Tr}_M N) = 0.$$
Hence, we get the equation (viii), and finish the proof.

$(3)$ Let
\(M = \bigoplus_{i=1}^n M_i,\)
where each \( M_i \) is indecomposable and \( M_i \ncong M_j \) for any \( i \neq j \).
It is easy to see
$\operatorname{Tr}_M N = \bigoplus_{i=1}^n \operatorname{Tr}_i$ and $M / \operatorname{Tr}_M N = \bigoplus_{i=1}^n M_i / \operatorname{Tr}_i,$
where \( \operatorname{Tr}_i = \operatorname{Tr}_{M_i} N \).

By definition, $M_i\in\operatorname{Gen}N$ if and only if $\operatorname{Tr}_i=M_i$, or equivalently, $M_i/\operatorname{Tr}_i=0$. Suppose that there is an index $i$ such that $M_i\notin\operatorname{Gen}N$. We first prove that $M_i/\operatorname{Tr}_i$ is indecomposable and nonzero.
Consider the short exact sequence
\begin{equation}\label{trdzhl}
\xymatrix{0 \ar[r]& \operatorname{Tr}_i \ar[r]& M_i \ar[r]^-{\pi_i}& M_i / \operatorname{Tr}_i \ar[r]& 0.}\end{equation}
Applying \( \operatorname{Hom}_A(M_i, -) \) to \eqref{trdzhl}, we get the exact sequence
\begin{equation}\label{trizhl}
\xymatrix{\operatorname{Hom}_A(M_i, M_i) \ar[r]^-{\pi_{i,\ast}}& \operatorname{Hom}_A(M_i, M_i / \operatorname{Tr}_i) \ar[r]& \operatorname{Ext}_A^1(M_i, \operatorname{Tr}_i).}\end{equation}
By (iv), the third term in \eqref{trizhl} is zero, so ${\pi_{i,\ast}}$ is surjective.
Applying \( \operatorname{Hom}_A(-, M_i / \operatorname{Tr}_i) \) to \eqref{trdzhl}, we get the exact sequence
\begin{equation}\label{trizhl2}
\xymatrix@C=1.6pc{0 \to \operatorname{Hom}_A(M_i / \operatorname{Tr}_i, M_i / \operatorname{Tr}_i) \ar[r]^-{\pi_i^\ast}& \operatorname{Hom}_A(M_i, M_i / \operatorname{Tr}_i) \ar[r]& \operatorname{Hom}_A(\operatorname{Tr}_i, M_i / \operatorname{Tr}_i).}
\end{equation}
By (v), the fourth term in \eqref{trizhl2} is zero, so $\pi_i^\ast$ is an isomorphism.
By \eqref{trizhl} and \eqref{trizhl2}, there exists a surjection
\begin{equation}\label{piman}
p_i:\operatorname{End}_A(M_i) \longrightarrow \operatorname{End}_A (M_i / \operatorname{Tr}_i)\end{equation}
such that there exists a morphism $y:M_i\to M_i$ satisfying $x\pi_i=\pi_iy$ for any morphism $x:M_i / \operatorname{Tr}_i \rightarrow M_i / \operatorname{Tr}_i$.
Since \( M_i \) is indecomposable, \( \operatorname{End}_A(M_i)\) is local. It follows that \( \operatorname{End}_A(M_i / \operatorname{Tr}_i) \) is also local, and thus \( M_i / \operatorname{Tr}_i \) is indecomposable.
The choice of $i$ gives $M_i\neq\operatorname{Tr}_i$, and thus $M_i/\operatorname{Tr}_i\neq0$.

We next prove that the nonzero quotients are pairwise non-isomorphic.
Suppose that there exists an isomorphism $h:M_i / \operatorname{Tr}_i \rightarrow M_j / \operatorname{Tr}_j$ for some $i\neq j$ such that $M_i,M_j\notin\operatorname{Gen}N$. Using similar arguments for getting the surjection $p_i$ in \eqref{piman}, we obtain a surjection
\[
\operatorname{Hom}_A(M_i, M_j) \longrightarrow \operatorname{Hom}_A(M_i / \operatorname{Tr}_i, M_j / \operatorname{Tr}_j)
\]
such that there exists a morphism $y:M_i\to M_j$ satisfying $x\pi_i=\pi_jy$ for any morphism $x:M_i / \operatorname{Tr}_i \rightarrow M_j / \operatorname{Tr}_j$.
In particular, for the morphism $h$ there is a morphism $g:M_i\to M_j$ satisfying $h\pi_i=\pi_jg$.
Thus, by the universal property of kernels, we get the following commutative diagram of exact sequences
$$\xymatrix{0\ar[r]& {\rm Tr}_i\ar[r]\ar@{-->}[d]^-f&M_i\ar[r]^-{\pi_i}\ar@{-->}[d]^-g&M_i/{{\rm Tr}_i}\ar[r]\ar[d]^-h&0\\
0\ar[r]& {\rm Tr}_j\ar[r]&M_j\ar[r]^-{\pi_j}&M_j/{{\rm Tr}_j}\ar[r]&0.}$$
Since $h$ is an isomorphism, we obtain a short exact sequence
\begin{equation}\label{mjtri}
\xymatrix{0 \ar[r]& \operatorname{Tr}_i \ar[r]& M_i \oplus \operatorname{Tr}_j \ar[r]& M_j \ar[r]& 0.}
\end{equation}
By (iv), we have $\operatorname{Ext}_A^1(M_j, \operatorname{Tr}_i) = 0$, so the exact sequence \eqref{mjtri} splits. It follows that
\(M_i \oplus \operatorname{Tr}_j \cong M_j \oplus \operatorname{Tr}_i.\) Since $M_i\ncong M_j$, we conclude \( M_i \) is a direct summand of
\( \operatorname{Tr}_i \), and then
\(\dim_kM_i \leq \dim_k\operatorname{Tr}_i.\)
Since \( \operatorname{Tr}_i \) is a submodule of \( M_i \), we have
\(\dim_k\operatorname{Tr}_i \leq \dim_kM_i.\)
Thus, \(\dim_k\operatorname{Tr}_i = \dim_kM_i\) and then $M_i = \operatorname{Tr}_i \in \operatorname{Gen} N$. This is a contradiction.
Hence, the nonzero $M_i/\operatorname{Tr}_i$ are pairwise non-isomorphic, and
\[
\left|M/\operatorname{Tr}_MN\right|
=\#\{i\mid M_i\notin\operatorname{Gen}N\}.
\]
Then the final assertion follows immediately, and we complete the proof.
\end{proof}

Now, we are in a position to give the main theorem of this section.
\begin{theorem} \label{3.3}
Let \( (A, T, B) \) be a splitting tilting triple.
Let \(M\oplus N\) be a basic \(\tau\)-tilting $A$-module with
\(N\in\mathcal{F}(T)\), and set $X:=M/\operatorname{Tr}_MN$. Assume that $X\in\mathcal{T}(T)$, no nonzero indecomposable direct summand of $M$ lies in $\operatorname{Gen}N$, and
$\operatorname{Hom}_A(X,\tau_{A,2}X)=0$. Then
\(
\operatorname{Hom}_A(T,X)\oplus\operatorname{Ext}_A^1(T,N)
\)
is a basic \( \tau_{2} \)-tilting $B$-module.
\end{theorem}
\begin{proof}
For convenience, set $W:=\operatorname{Hom}_A(T, M / \operatorname{Tr}_M N)$ and $E:=\operatorname{Ext}_A^1(T, N)$.

\noindent{\bf Step 1:} \(\operatorname{Hom}_B(W \oplus E, \tau_{B,2}(W \oplus E)) = 0.\)

Since \( M \oplus N \) is \( \tau_A \)-tilting, we have
\(\operatorname{Hom}_A(M, \tau_A N) = 0.\)
Thus, \(\operatorname{Hom}_A(X,\tau_A N)=0\).
By assumption, $M/\operatorname{Tr}_MN\in\mathcal{T}(T)$.
Replacing \(M\) by \(X\) in Lemma \ref{3.1} and using the condition
$\operatorname{Hom}_A(X,\tau_{A,2}X)=0$, we finish the proof of Step 1.

\noindent{\bf Step 2:}
\(\operatorname{Ext}_B^1(W \oplus E, W \oplus E) = 0.\)

By Lemma \ref{xt}, Lemma \ref{3.2}(vi),(vii), and
Lemma \ref{splitting}(4), we have
\begin{align*}
\operatorname{Ext}_B^1(W,W)
&\cong \operatorname{Ext}_A^1(M / \operatorname{Tr}_M N,
 M / \operatorname{Tr}_M N)=0,\\
\operatorname{Ext}_B^1(E,E)
&\cong \operatorname{Ext}_A^1(N,N)=0,\\
\operatorname{Ext}_B^1(W,E)
&\cong \operatorname{Ext}_A^2(M / \operatorname{Tr}_M N,N)=0,\\
\operatorname{Ext}_B^1(E,W)
&\cong \operatorname{Hom}_A(N,M / \operatorname{Tr}_M N)=0.
\end{align*}
Hence, we finish the proof of Step 2.

\noindent{\bf Step 3:}
\( W \oplus E \) is basic and \(|W \oplus E| = |B|.\)

By Theorem \ref{BB}, the two Brenner--Butler equivalences preserve isomorphisms and indecomposable properties of objects. Moreover, $W\in\mathcal{Y}(T)$ and $E\in\mathcal{X}(T)$, so they have no nonzero indecomposable direct summand in common. Hence, by Lemma \ref{3.2}(3), we get that $W\oplus E$ is basic and
\[
|W\oplus E|
=|M/\operatorname{Tr}_MN|+|N|
=|M|+|N|
=|M\oplus N|
=|A|
=|B|.
\]
The last equality also follows from the derived equivalence \eqref{RHom}, which identifies the ranks of the Grothendieck groups.
Therefore, we complete the proof.
\end{proof}

The following proposition shows that the two corresponding hypotheses in Theorem \ref{3.3} are the best possible conditions for the transferred $\tau_{B,2}$-orthogonality and for the number of indecomposable direct summands, respectively.

\begin{proposition}\label{sharp-transfer}
Let $(A,T,B)$ be a splitting tilting triple, let $X\in\mathcal{T}(T)$, and set
$W=\Hom_A(T,X)$. Then
\[
\Hom_B(W,\tau_{B,2}W)=0
\quad\Longleftrightarrow\quad
\Hom_A(X,\tau_{A,2}X)=0.
\]
Moreover, let $M\oplus N$ be a basic $\tau$-tilting $A$-module such that
$N\in\mathcal{F}(T)$ and $X=M/\operatorname{Tr}_MN\in\mathcal{T}(T)$, and set $W=\Hom_A(T,X)$,
$E=\Ext_A^1(T,N)$. Then
\[
|W\oplus E|=|B|
\quad\Longleftrightarrow\quad
\text{no nonzero indecomposable direct summand of $M$ lies in $\operatorname{Gen} N$}.
\]
\end{proposition}
\begin{proof}
The first equivalence follows directly from Lemma \ref{3.1}$(2)$. For the second assertion, Lemma \ref{3.2}(3) gives
\[
|X|=\#\{i\mid M_i\notin\operatorname{Gen} N\},
\]
where $M=\bigoplus_iM_i$ is basic. The Brenner--Butler equivalences preserve indecomposables, while $W\in\mathcal{Y}(T)$ and $E\in\mathcal{X}(T)$ have no nonzero direct summand in common. Hence
$|W\oplus E|=|X|+|N|$. Since $|M|+|N|=|A|=|B|$, the desired equivalence follows.
\end{proof}

In the following, for any bound quiver algebra $A=kQ/I$ over a field $k$, we denote by $S_i$ the simple $A$-module corresponding to the vertex $i$ of $Q$, and denote by $P_i$ and $I_i$ the projective cover and injective envelope of $S_i$, respectively. For any $A$-module $M$, denote by $\operatorname{rad}M$ and $\operatorname{soc}M$ the radical and socle of $M$, respectively.
Now let us provide an example illustrating Theorem \ref{3.3} as follows.
\begin{example}\label{genuine-example}
Let $A=kQ/I$, where
\[
Q:\quad 1\overset{a_1}\longrightarrow2\overset{a_2}\longrightarrow3\overset{a_3}\longrightarrow4\overset{a_4}\longrightarrow5
\]
and $I=<a_1a_2a_3,a_3a_4>$.
Set $T=P_1\oplus S_2\oplus P_2\oplus P_4\oplus P_5$.
There are exact sequences
\[
0\longrightarrow P_3\longrightarrow P_2\longrightarrow S_2\longrightarrow0
\]
and
\[
0\longrightarrow A\longrightarrow
P_1\oplus P_2^{\oplus2}\oplus P_4\oplus P_5
\longrightarrow S_2\longrightarrow0.
\]
By the Auslander--Reiten formula, we get $\Ext_A^1(S_2,T)=0$. Thus, $T$ is a classical tilting module. Moreover, we have
\[
  \mathcal{T}(T)=\operatorname{add}(P_5\oplus P_4 \oplus S_4\oplus P_2\oplus \text{rad}P_1\oplus S_2\oplus P_1\oplus I_2\oplus S_1)~\text{and}~ \mathcal{F}(T)=\operatorname{add}S_3.
\]
By the exact sequence
\[
0\longrightarrow S_3\longrightarrow I_3
\longrightarrow I_2\longrightarrow0,
\]
we have $\operatorname{id}_AS_3=1$. Thus, Lemma \ref{splitting} shows that
$(A,T,B)$ with $B=\End_A(T)$ is a splitting tilting triple.

Let $L=\text{rad}P_1$, and take
$M=P_1\oplus L\oplus P_2\oplus P_5$ and $N=S_3$. Clearly,
$M\in\mathcal{T}(T)$ and $N\in\mathcal{F}(T)$. The
Auslander--Reiten translations of $L$ and $S_3$ are $\tau_AL=P_3$ and $\tau_AS_3=S_4$, respectively,
and it is easy to see that
$\Hom_A(M\oplus N,P_3\oplus S_4)=0$.
Since $|M\oplus N|=5=|A|$, the module $M\oplus N$ is a basic $\tau$-tilting module. On the other hand, by the exact sequence
\[
0\longrightarrow P_5\longrightarrow P_4\longrightarrow P_2
\longrightarrow L\longrightarrow0~\text{and}~
0\longrightarrow P_5\longrightarrow P_4\longrightarrow P_3
\longrightarrow S_3\longrightarrow0,
\]
we conclude that $\operatorname{pd}_AL=\operatorname{pd}_AS_3=2$. Thus, $M\oplus N$ is not a classical tilting module.

We next calculate the trace by its definition. The only nonzero morphisms from $S_3$ to the summands of $M$ are
\[
S_3\lhook\joinrel\longrightarrow\operatorname{soc}P_1~~\text{and}~~
S_3\lhook\joinrel\longrightarrow\operatorname{soc}L.
\]
Thus, we obtain
\[
\operatorname{Tr}_MN
=S_3\oplus S_3\subseteq P_1\oplus L
\]
and
\[
X:=M/\operatorname{Tr}_MN
=I_2\oplus S_2\oplus P_2\oplus P_5.
\]
No nonzero indecomposable summand of $M$ lies in
$\operatorname{Gen} N=\operatorname{add}S_3$. Furthermore,
\[
\Omega_AI_2=S_3,\qquad \Omega_AS_2=P_3,
\]
and then
\[
\tau_{A,2}X=S_4\neq0,\qquad \Hom_A(X,S_4)=0.
\]
Thus, every condition of Theorem \ref{3.3} is satisfied nontrivially. Hence, we conclude that  \(
\operatorname{Hom}_A(T,X)\oplus\operatorname{Ext}_A^1(T,N)
\)
is a basic \( \tau_{B,2} \)-tilting module.

Now let us check this  result  directly over $B$. We denote the summands of $T$ by
\[
T_1=P_1,\quad T_2=S_2,\quad T_3=P_2,\quad
T_4=P_4,\quad T_5=P_5.
\]
The non-identity irreducible morphisms in $\operatorname{add}T$ are
\[
T_3\longrightarrow T_1,\qquad
T_3\longrightarrow T_2,\qquad
T_4\longrightarrow T_3,\qquad
T_5\longrightarrow T_4,
\]
and every composition of two consecutive morphisms is zero. Since we use right
$B$-modules, the quiver $Q_B$ of $B$ is
\[
\xymatrix@=0.4cm{1'\ar[dr]&&&\\
&3'\ar[r]&4'\ar[r]&5'\\
2'\ar[ur]&&&
}
\]
and $B\cong kQ_B/J_{Q_B}^2$, where $J_{Q_B}$ is the ideal of $kQ_B$ generated by all arrows of $Q_B$.

A direct computation shows that
\(\Hom_A(T,X)\cong
I_{3'}\oplus P_{2'}\oplus P_{3'}\oplus P_{5'}\) and
\(\Ext_A^1(T,S_3)\cong S_{2'}\). Thus, the module constructed in Theorem \ref{3.3} is
\[
H:=\Hom_A(T,X)\oplus \Ext_A^1(T,S_3)
\cong I_{3'}\oplus P_{2'}\oplus P_{3'}\oplus P_{5'}\oplus S_{2'}.
\]

By the Auslander--Reiten formula, $\Ext_B^1(H,H)=0$.
Moreover, $\Omega_BI_{3'}\cong S_{3'}$,
$\Omega_B S_{2'}\cong S_{3'}$
and the almost split sequence
\[
0\longrightarrow S_{4'}\longrightarrow P_{3'}
\longrightarrow S_{3'}\longrightarrow0
\]
gives $\tau_B S_{3'}\cong S_{4'}$. Hence,
$\tau_{B,2}H\cong S_{4'}\oplus S_{4'}$. There are no nonzero morphisms from
any indecomposable summand of $H$ to $S_{4'}$, and hence
$\Hom_B(H,\tau_{B,2}H)=0$.
Finally, the five  summands of $H$ are pairwise non-isomorphic, so
$H$ is basic and $|H|=5=|B|$. This directly verifies that $H$ is a basic
$\tau_{B,2}$-tilting module.

It is worth noting that the above module $H$ is not a
$2$-tilting module. Indeed, the following minimal projective resolution
\[
0\longrightarrow P_{5'}\longrightarrow P_{4'}
\longrightarrow P_{3'}\longrightarrow P_{1'}\oplus P_{2'}
\longrightarrow I_{3'}\longrightarrow 0
\]
shows that $\operatorname{pd}_BI_{3'}=3$.
\end{example}

\begin{corollary}\label{splitting-1-tilting}
Let $(A,T,B)$ be a splitting tilting triple. Let $M\oplus N$ be a basic
$1$-tilting $A$-module with
$N\in\mathcal{F}(T)$, and set $X:=M/\operatorname{Tr}_MN$.
Assume that $X\in\mathcal{T}(T)$, no nonzero indecomposable direct
summand of $M$ lies in $\operatorname{Gen}N$, and
$\operatorname{Hom}_A(X,\tau_{A,2}X)=0$. Then
\[
H:=\operatorname{Hom}_A(T,X)\oplus\operatorname{Ext}_A^1(T,N)
\]
is a basic  $2$-tilting
$B$-module if and only if
\(
\operatorname{pd}_BH\leq2
\)
and there is an exact sequence
\[
0\longrightarrow B\longrightarrow H^0\longrightarrow H^1
\longrightarrow H^2\longrightarrow0~\text{with~each}~
H^i\in\operatorname{add}H.
\]
\end{corollary}
\begin{proof}
Note that every $1$-tilting module is $\tau$-tilting. By
Definition \ref{tau-n} and Lemma \ref{tau-gen},
$\operatorname{Ext}_B^1(H,H)=0=\operatorname{Ext}_B^2(H,H)$.
The equivalence is exactly the conditions {\rm(T1)} and
{\rm(T3)} in Definition \ref{ntilting}.
\end{proof}

The following gives an example illustrating Corollary \ref{splitting-1-tilting} with $A$ non-hereditary.

\begin{example}\label{nonhereditary-splitting-example}
Let $A=kQ/I$, where
\[
Q:\quad 1\overset{a_1}\longrightarrow2\overset{a_2}\longrightarrow3\overset{a_3}\longrightarrow4\overset{a_4}\longrightarrow5
\]
and $I=<a_1a_2a_3>$. Then $A$ is not hereditary. Let
$T=P_1\oplus S_2\oplus P_2\oplus P_4\oplus P_5$.
By the exact sequence
\begin{equation}\label{exttzhl}
0\longrightarrow P_3\longrightarrow P_2\longrightarrow S_2
\longrightarrow0,\end{equation}
we have $\operatorname{pd}_AT\leq1$. Applying $\operatorname{Hom}_A(-,T)$ to \eqref{exttzhl} and noting that
${\rm dim}_k\Hom_A(S_2,T)=1$, ${\rm dim}_k\Hom_A(P_2,T)=3$ and ${\rm dim}_k\Hom_A(P_3,T)=2$, we get $\Ext^1_A(S_2,T)=0$, and then $\operatorname{Ext}_A^1(T,T)=0$.
Noting that we have the exact sequence
\[
0\longrightarrow A\longrightarrow
P_1\oplus P_2^{\oplus2}\oplus P_4\oplus P_5
\longrightarrow S_2\longrightarrow0,
\]
we conclude that $T$ is a classical tilting
module. A direct calculation gives
\[
\mathcal{T}(T)=\text{mod}A/\text{add}(P_3\oplus P_3/\text{soc}P_3\oplus S_3)~\text{and}~ \mathcal{F}(T)=\operatorname{add}S_3.
\]
By the exact sequence
$0\to S_3\to I_3\to I_2
\to 0$,
we get $\operatorname{id}_AS_3=1$, and then Lemma
\ref{splitting} shows that $(A,T,B)$ with $B=\operatorname{End}_A(T)$
is a splitting tilting triple.

Set $L=\operatorname{rad}P_1$,
$M=P_1\oplus L\oplus P_2\oplus P_5$ and $N=S_3$. Then $M\in\mathcal{T}(T)$ and $N\in\mathcal{F}(T)$. The exact
sequences
\[
0\longrightarrow P_4\longrightarrow P_2\longrightarrow L
\longrightarrow0~\text{and}~
0\longrightarrow P_4\longrightarrow P_3\longrightarrow S_3
\longrightarrow0
\]
give $\operatorname{pd}_A(M\oplus N)\leq1$ and
$\operatorname{Ext}_A^1(M\oplus N,M\oplus N)=0$. Since $M\oplus N$
has five pairwise non-isomorphic indecomposable direct summands, it is a
basic $1$-tilting $A$-module.

The only nonzero morphisms from $S_3$ to the summands of $M$ are the embeddings
into the socles of $P_1$ and $L$. Thus, we get
\[
\operatorname{Tr}_MN
=S_3\oplus S_3\subseteq P_1\oplus L~\text{and}~
X:=M/\operatorname{Tr}_MN
=I_2\oplus S_2\oplus P_2\oplus P_5.
\]
No nonzero indecomposable direct summand of $M$ belongs to
$\operatorname{Gen}N=\operatorname{add}S_3$. Furthermore,
$\Omega_AI_2=S_3$ and $\Omega_AS_2=P_3$, and hence
$\tau_{A,2}X=S_4$ and
$\operatorname{Hom}_A(X,\tau_{A,2}X)=0$.

The quiver $Q_B$ of $B$ is
\[
\xymatrix@=0.45cm{1'\ar[dr]^{b_1}&&&\\
&3'\ar[r]^{b_3}&4'\ar[r]^{b_4}&5'\\
2'\ar[ur]_{b_2}&&&}
\]
and
$B\cong kQ_B/\langle b_1b_3,b_2b_3\rangle$.
A direct calculation gives
\[
H:=\operatorname{Hom}_A(T,X)\oplus\operatorname{Ext}_A^1(T,N)\cong I_{3'}\oplus P_{2'}\oplus P_{3'}\oplus P_{5'}\oplus S_{2'}.
\]
The minimal projective resolutions
\[
0\longrightarrow P_{4'}\longrightarrow P_{3'}
\longrightarrow P_{1'}\oplus P_{2'}\longrightarrow I_{3'}
\longrightarrow0
\]
and
\[
0\longrightarrow P_{4'}\longrightarrow P_{3'}
\longrightarrow P_{2'}\longrightarrow S_{2'}\longrightarrow0
\]
show that $\operatorname{pd}_BH=2$. Moreover, we have the following exact sequence
\[
0\longrightarrow B\longrightarrow
I_{3'}\oplus P_{2'}\oplus P_{3'}^{\oplus2}\oplus P_{5'}
\longrightarrow S_{2'}\oplus P_{2'}
\longrightarrow S_{2'}\longrightarrow0,
\]
where the three terms after $B$ belong to $\operatorname{add}H$.  Therefore, by Corollary
\ref{splitting-1-tilting}, we get $H$ is a basic $2$-tilting
$B$-module.
\end{example}

As an application, we recover the following main result of \cite{PENG}.
\begin{corollary}\label{3.4}\textup{(\cite[Theorem 2.6]{PENG})}
Let \( (A, T, B) \) be a tilting triple with $A$ hereditary.
Let \( M \oplus N \) be a basic $1$-tilting $A$-module with \( M \in \mathcal{T}(T) \) and \( N \in \mathcal{F}(T) \).
Assume that no nonzero indecomposable direct summand of \( M \) lies in \( \operatorname{Gen} N \). Then
\(\operatorname{Hom}_A(T, M / \operatorname{Tr}_M N) \oplus \operatorname{Ext}_A^1(T, N)\)
is a basic $2$-tilting $B$-module.
\end{corollary}

In the remainder of this section, assume that $(A,T,B)$ is a
splitting tilting triple. For any $A$-modules $M$ and $N$, write $[M,N]$ for the componentwise
isomorphism class of the ordered pair $(M,N)$. Set
\[
\begin{split}
\mathfrak D_{\tau_2}(A,T):=\bigl\{[M,N]\ \bigm|\;&
M\oplus N\text{ is a basic }\tau_A\text{-tilting module},\  N\in\mathcal F(T),\\
&X_{M,N}:=M/\operatorname{Tr}_MN\in\mathcal T(T),\\
&\text{no nonzero indecomposable direct summand}\\
&\text{of }M\text{ lies in }\operatorname{Gen}N,\\
&\operatorname{Hom}_A(X_{M,N},\tau_{A,2}X_{M,N})=0
\bigr\}
\end{split}
\]
and let $\mathfrak T_{\tau_2}(B)$ be the set of isomorphism classes
$[H]$ of basic $\tau_{2}$-tilting $B$-modules. By Theorem \ref{3.3}, we have
the following well-defined map:
\begin{equation}\label{phi}
\begin{split}
\Phi:\mathfrak D_{\tau_2}(A,T)\longrightarrow
\mathfrak T_{\tau_2}(B),\quad
[M,N]\longmapsto
\bigl[\operatorname{Hom}_A(T,X_{M,N})
      \oplus\operatorname{Ext}_A^1(T,N)\bigr].
\end{split}
\end{equation}

\begin{proposition}\label{phi-injective}
The map $\Phi$ in \eqref{phi} is injective.
\end{proposition}
\begin{proof}
Let $[M_i,N_i]\in\mathfrak D_{\tau_2}(A,T)$ and set
$X_i:=M_i/\operatorname{Tr}_{M_i}N_i$ for $i=1,2$. Suppose that $\Phi([M_1,N_1])=\Phi([M_2,N_2])$. The uniqueness of the
decomposition with respect to the torsion pair
$(\mathcal X(T),\mathcal Y(T))$ gives
\[
\operatorname{Hom}_A(T,X_1)\cong\operatorname{Hom}_A(T,X_2)
\quad\text{and}\quad
\operatorname{Ext}_A^1(T,N_1)\cong\operatorname{Ext}_A^1(T,N_2).
\]
Then the two Brenner--Butler equivalences give
$X_1\cong X_2$ and $N_1\cong N_2$. We identify these modules and write
them as $X$ and $N$, respectively.

For $i=1,2$, let
$\mathcal G_i:=\operatorname{Gen}(M_i\oplus N)$. Since $M_i\oplus N$ is $\tau_A$-tilting, $\mathcal G_i$ is a torsion
class by \cite[Theorem 2.7]{ADA}. It contains both $X$ and $N$.
Consequently, if $\operatorname{Tors}(X\oplus N)$ denotes the smallest
torsion class containing $X\oplus N$, then
$\operatorname{Tors}(X\oplus N)\subseteq\mathcal G_i$. Conversely,
$\operatorname{Tr}_{M_i}N\in\operatorname{Gen}N
\subseteq\operatorname{Tors}(X\oplus N)$, and the exact sequence
\[
0\longrightarrow\operatorname{Tr}_{M_i}N
\longrightarrow M_i\longrightarrow X\longrightarrow0
\]
shows, by extension closure, that
$M_i\in\operatorname{Tors}(X\oplus N)$. Thus,
$\mathcal G_i\subseteq\operatorname{Tors}(X\oplus N)$. Hence, $\mathcal G_1=\mathcal G_2$. The bijection (cf. \cite[Theorem 2.7]{ADA}) between basic support $\tau$-tilting modules and functorially finite torsion classes
yields
\(
\operatorname{add}(M_1\oplus N)=\operatorname{add}(M_2\oplus N).
\)
Since $M_1\oplus N$ and $M_2\oplus N$ are basic, by the Krull--Schmidt theorem, we get
$M_1\cong M_2$. Therefore, $[M_1,N_1]=[M_2,N_2]$ and we finish the proof.
\end{proof}

Let $S$ be a simple non-injective $A$-module and let $P$ be the direct
sum of one representative of every indecomposable projective
$A$-module other than the projective cover $P(S)$ of $S$. The module
$T=\tau_A^{-1}S\oplus P$ is called a \emph{BB-tilting module} if
\(
\operatorname{pd}_A\tau_A^{-1}S\leq1\)
 and
\(\operatorname{Ext}_A^1(S,S)=0.
\)
Under these conditions, $T$ is a classical $1$-tilting module. Since $S$ is simple, we have
$\operatorname{Gen}S=\operatorname{add}S$. Then, since
$\operatorname{Ext}_A^1(S,S)=0$, by Lemma \ref{tau-gen}, we have
$\operatorname{Hom}_A(\tau_A^{-1}S,S)\cong\operatorname{Hom}_A(S,\tau_AS)=0$.

\begin{lemma}\label{unique}
Let $T=\tau_A^{-1}S\oplus P$ be a BB-tilting module and
$B=\operatorname{End}_A T$. Then
\(
\mathcal F(T)=\operatorname{add}S\) and
\(\mathcal X(T)=\operatorname{add}\operatorname{Ext}_A^1(T,S).
\)
In particular, $S$ and $\operatorname{Ext}_A^1(T,S)$ are, up to
isomorphism, the unique indecomposable objects of $\mathcal F(T)$ and
$\mathcal X(T)$, respectively.
\end{lemma}
\begin{proof}
By the definition of {\em BB}-tilting modules, we get
$\operatorname{Hom}_A(T,S)=0$. Thus, we have
$\operatorname{add}S\subseteq\mathcal F(T)$. Conversely, let
$L\in\mathcal F(T)$, i.e. $\Hom_A(T,L)=0$, thus $\operatorname{Hom}_A(P,L)=0$. By the definition of $P$, every composition
factor of $L$ is isomorphic to $S$. Since
$\operatorname{Ext}_A^1(S,S)=0$, we get
$L\in\operatorname{add}S$. Hence,
$\mathcal F(T)=\operatorname{add}S$. The second equality follows from the Brenner--Butler equivalence
\(
\operatorname{Ext}_A^1(T,-):
\mathcal F(T)\xrightarrow{\sim}\mathcal X(T).
\)

\end{proof}

\begin{proposition}\label{phi-surjective-bb}
Let $T=\tau_A^{-1}S\oplus P$ be a splitting BB-tilting module and
$B=\operatorname{End}_A T$. Let $W\oplus E$ be a basic
$\tau_{B,2}$-tilting module with
\(
W\in\mathcal Y(T)\), \( E\in\mathcal X(T)\). Assume  that
\(\operatorname{Hom}_B(W,\tau_BW)=0.\)
Then there exists $[M,N]\in\mathfrak D_{\tau_2}(A,T)$ with
$M\in\mathcal T(T)$, $N\in\mathcal F(T)$ such that
\(
\Phi([M,N])=[W\oplus E].
\)
\end{proposition}
\begin{proof}
By Lemma \ref{unique}, $\mathcal X(T)
=\operatorname{add}\operatorname{Ext}_A^1(T,S)$. Since $E$ is basic, we conclude that
either $E=0$ or $E\cong\operatorname{Ext}_A^1(T,S)$.

Suppose that $E=0$. By the Brenner--Butler equivalence, there
exists a basic module $M\in\mathcal T(T)$ such that
$W\cong\operatorname{Hom}_A(T,M)$. The equivalence preserves
indecomposable summands, so $|M|=|W|=|B|=|A|$. By Lemma \ref{3.1}(1), we get
$\operatorname{Hom}_A(M,\tau_AM)\cong\operatorname{Hom}_B(W,\tau_BW)=0$, i.e. $M$ is $\tau_A$-tilting. Since in this case $\operatorname{Hom}_A(T,M)$ is $\tau_{B,2}$-tilting, by Lemma \ref{3.1}(2), we get $\operatorname{Hom}_A(M,\tau_{A,2}M)=0$. Thus, we obtain that $[M,0]\in\mathfrak D_{\tau_2}(A,T)$,
and $\Phi([M,0])=[W]$.

Suppose that $E\cong\operatorname{Ext}_A^1(T,S)$. Again, the Brenner--Butler
equivalence gives a basic $M\in\mathcal T(T)$ such that
$W\cong\operatorname{Hom}_A(T,M)$. Moreover, we have
\begin{equation}\label{zhxgss}|M\oplus S|=|W\oplus E|=|B|=|A|.\end{equation}
The hypotheses on $W\oplus E$ imply
\begin{equation*}
\begin{split}
&\operatorname{Hom}_B(\operatorname{Hom}_A(T,M),\tau_B \operatorname{Hom}_A(T,M))=0,\\
&\operatorname{Hom}_B(\operatorname{Hom}_A(T,M),\tau_{B,2} \operatorname{Hom}_A(T,M))=0,\\
&\operatorname{Hom}_B(\operatorname{Hom}_A(T,M),\tau_{B,2} \operatorname{Ext}_A^1(T,S))=0,\\
&\operatorname{Ext}_B^1(\operatorname{Ext}_A^1(T,S),\operatorname{Hom}_A(T,M))=0.
\end{split}
\end{equation*}
By Lemma \ref{xt} and Lemma \ref{3.1}, we have
\begin{equation}\label{5=0}
\begin{split}
&\operatorname{Hom}_A(M,\tau_A M)=0,~
\operatorname{Hom}_A(M,\tau_{A,2} M)=0,\\
&\operatorname{Hom}_A(M,\tau_A S)=0,~
\operatorname{Hom}_A(S,M)=0.
\end{split}
\end{equation}

Assume that $\operatorname{Hom}_A(S,\tau_AM)=0$. Since $\Hom_A(S,\tau_AS)=0$, by \eqref{5=0}, we get $M\oplus S$ is $\tau_A$-rigid. By \eqref{zhxgss}, we conclude that $M\oplus S$ is $\tau_A$-tilting. Since $\operatorname{Hom}_A(S,M)=0$, we get
$\operatorname{Tr}_MS=0$ and $X_{M,S}=M\in\mathcal T(T)$. Since
$\operatorname{Gen}S=\operatorname{add}S$ and
$\operatorname{Hom}_A(S,M)=0$, no nonzero indecomposable summand of
$M$ lies in $\operatorname{Gen}S$. Thus,
$[M,S]\in\mathfrak D_{\tau_2}(A,T)$ and
\[
\Phi([M,S])
=\bigl[\operatorname{Hom}_A(T,M)\oplus\operatorname{Ext}_A^1(T,S)\bigr]
=[W\oplus E].
\]
Hence, the assertion holds in this case.

Now, assume that $\operatorname{Hom}_A(S,\tau_AM)\neq0$. Since
$\operatorname{Gen}S=\operatorname{add}S$, by Lemma \ref{tau-gen}, we have
$\operatorname{Ext}_A^1(M,S)\neq0$. Write
$M=\bigoplus_{i=1}^nM_i$, where the $M_i$ are pairwise non-isomorphic and indecomposable.
For each $1\leq i\leq n$, since $\operatorname{End}_A(S)$ is a division algebra, choose a basis
$e_{i1},\ldots,e_{ir_i}$ of $\operatorname{Ext}_A^1(M_i,S)$ viewed as a left $\operatorname{End}_A(S)$-module defined by pushouts. Consider the following exact sequence
\begin{equation}\label{sss}
0\longrightarrow S^{\oplus r_i}
\longrightarrow M_i'\longrightarrow M_i\longrightarrow0
\end{equation}
such that the connecting morphism
$\delta_i:\operatorname{Hom}_A(S^{\oplus r_i},S)
\longrightarrow\operatorname{Ext}_A^1(M_i,S)$
is surjective.
Set $r:=\sum_{i=1}^nr_i$ and $M':=\bigoplus_{i=1}^nM_i'$.
Taking the direct sum of the exact sequences \eqref{sss}, we obtain the exact sequence
\begin{equation}\label{tr}
0 \longrightarrow S^{\oplus r} \longrightarrow M'\longrightarrow M \longrightarrow 0.
\end{equation}
Note that $\operatorname{Hom}_A(M,S)=0$ since $M\in\mathcal T(T)$ and $S\in\mathcal F(T)$. Applying $\operatorname{Hom}_A(-,S)$ to \eqref{tr}, we get the exact sequence
\[\xymatrix{0\ar[r]&
\operatorname{Hom}_A(M',S)\ar[r]&
\operatorname{Hom}_A(S^{\oplus r},S)
\ar[r]^-\delta&\operatorname{Ext}_A^1(M,S)
\ar[r]&\operatorname{Ext}_A^1(M',S)\ar[r]&0,}\]
where $\delta=\bigoplus_{i=1}^n\delta_i$ is surjective.
Thus, we get $\operatorname{Ext}_A^1(M',S)=0$. Noting that $${\rm dim}_k\operatorname{Hom}_A(S^{\oplus r},S)={\rm dim}_k\operatorname{Ext}_A^1(M,S)=rd_S,$$ where $d_S={\rm dim}_k\operatorname{End}_A(S)$, we obtain $\operatorname{Hom}_A(M',S)=0$. Since $\mathcal F(T)=\operatorname{add}S$, we have $M'\in\mathcal T(T)$. Applying $\operatorname{Hom}_A(T,-)$ to \eqref{tr}, we get the exact sequence
\begin{equation}\label{mms}
0 \longrightarrow \operatorname{Hom}_A(T,M') \longrightarrow \operatorname{Hom}_A(T,M) \longrightarrow {\operatorname{Ext}^1_A(T,S^{\oplus r})} \longrightarrow 0.
\end{equation}
Let $L\in\operatorname{Gen}M$. Then $L\in\mathcal T(T)$. Applying
$\operatorname{Hom}_B(-,\operatorname{Hom}_A(T,L))$ to \eqref{mms},
we get the exact sequence
\[
\begin{split}
\operatorname{Ext}_B^1\bigl(\operatorname{Hom}_A(T,M),
  \operatorname{Hom}_A(T,L)\bigr)
&\longrightarrow
\operatorname{Ext}_B^1\bigl(\operatorname{Hom}_A(T,M'),
  \operatorname{Hom}_A(T,L)\bigr)\\
&\longrightarrow
\operatorname{Ext}_B^2\bigl(\operatorname{Ext}_A^1(T,S^{\oplus r}),
  \operatorname{Hom}_A(T,L)\bigr).
\end{split}
\]
By Lemma \ref{xt}(1) and (4), we have the exact sequence
\[
\operatorname{Ext}_A^1(M,L)
\longrightarrow
\operatorname{Ext}_A^1(M',L)
\longrightarrow
\operatorname{Ext}_A^1(S^{\oplus r},L).
\]
The first and third terms vanish by \eqref{5=0} and
Lemma \ref{tau-gen}. Thus
$\operatorname{Ext}_A^1(M',L)=0$. Since $L$ was arbitrary,
we get $\operatorname{Ext}_A^1(M',\operatorname{Gen}M)=0$ and then
$\operatorname{Hom}_A(M,\tau_AM')=0$.
Applying $\operatorname{Hom}_A(-,\tau_A S)$ to \eqref{tr}, we get the exact sequence
\begin{equation*}
\operatorname{Hom}_A(M,\tau_A S) \longrightarrow \operatorname{Hom}_A(M',\tau_A S) \longrightarrow \operatorname{Hom}_A(S^{\oplus r},\tau_A S).
\end{equation*}
Noting that $\operatorname{Hom}_A(M,\tau_AS)=0$ and
$\operatorname{Hom}_A(S,\tau_AS)=0$, we get $\operatorname{Hom}_A(M',\tau_AS)=0$.
Similarly, applying $\operatorname{Hom}_A(-,\tau_AM')$ to \eqref{tr}, we have the exact sequence
\[
\operatorname{Hom}_A(M,\tau_AM')
\longrightarrow
\operatorname{Hom}_A(M',\tau_AM')
\longrightarrow
\operatorname{Hom}_A(S^{\oplus r},\tau_AM').
\]
Since
$\operatorname{Ext}_A^1(M',S)=0$ and
$\operatorname{Gen}S=\operatorname{add}S$, we get
$\operatorname{Ext}_A^1(M',\operatorname{Gen}S)=0$. By Lemma \ref{tau-gen}, we have
$\operatorname{Hom}_A(S,\tau_AM')=0$. It follows that
$\operatorname{Hom}_A(M',\tau_AM')=0$. Hence, we conclude that
$\operatorname{Hom}_A(M'\oplus S,\tau_A(M'\oplus S))=0$.

In what follows, it remains to prove that $M'$ is basic, i.e. the $M_i'$ are
pairwise non-isomorphic and indecomposable.

For any $1\leq j\leq n$, applying $\operatorname{Hom}_A(-,M_j)$ to the exact sequence \eqref{sss}, we obtain the exact sequence
\begin{equation*}
0 \longrightarrow \operatorname{Hom}_A(M_i,M_j) \longrightarrow \operatorname{Hom}_A(M_i',M_j)\longrightarrow \operatorname{Hom}_A(S^{\oplus r_i},M_j).
\end{equation*}
Since $\operatorname{Hom}_A(S,M)=0$, we get $\operatorname{Hom}_A(M_i,M_j) \cong \operatorname{Hom}_A(M_i',M_j)$.

Next, let us prove that $M_i'$ is indecomposable. In fact, for any
$\phi\in\operatorname{End}_A(M_i')$,
the isomorphism
$\operatorname{Hom}_A(M_i,M_i)
\xrightarrow{\sim}
\operatorname{Hom}_A(M_i',M_i)$
gives $\psi\in\operatorname{End}_A(M_i)$ such that
$\psi g_i=g_i\phi$, i.e. we have the following commutative diagram
\begin{equation*}
\begin{tikzcd}
	0 & {S^{\,\oplus\,r_i}} & {M_i^{\,\prime}} & {M_i} & 0 \\
	&& {M_i^{\,\prime}} & {M_i}
	\arrow[from=1-1, to=1-2]
	\arrow["{f_i}", from=1-2, to=1-3]
	\arrow["{g_i}", from=1-3, to=1-4]
	\arrow["\phi"', from=1-3, to=2-3]
	\arrow[from=1-4, to=1-5]
	\arrow["\psi", dashed, from=1-4, to=2-4]
	\arrow["{g_i}"', from=2-3, to=2-4]
\end{tikzcd}
\end{equation*}
Since $M_i$ is indecomposable, $\psi$ is either
invertible or nilpotent. Since $\operatorname{Hom}_A(M',S)=0$ and $\operatorname{Ext}_A^1(M',S)=0$, we have
$\operatorname{Hom}_A(M_i',S^{\oplus r_i})=0$ and $\operatorname{Ext}_A^1(M_i',S^{\oplus r_i})=0$.
Suppose $\psi$ is invertible. Applying $\Hom_A(M_i',-)$ to the exact sequence \eqref{sss}, we obtain that there exists $\phi'\in\operatorname{End}_A(M_i')$ such that $g_i\phi'=\psi^{-1}g_i$. Then we get
$\phi'\phi-\operatorname{id}_{M_i'}$ and $\phi\phi'-\operatorname{id}_{M_i'}$ factor through
$S^{\oplus r_i}$, and then equal to zero. Hence, $\phi$
is invertible.
Suppose that $\psi$ is nilpotent. Then some power of $\phi$
factors through $S^{\oplus r_i}$ and thus is zero, i.e. $\phi$ is nilpotent. Hence, every
endomorphism of $M_i'$ is invertible or nilpotent, it follows that $M_i'$ is indecomposable.

Assume that there exists an isomorphism $\phi:M_i'\to M_j'$ for some $i\ne j$. Uisng the isomorphism $\operatorname{Hom}_A(M_i,M_j) \cong \operatorname{Hom}_A(M_i',M_j)$, there is a morphism $\psi$ such that the right square of the following diagram is commutative
 \begin{equation}\label{mij}
\begin{tikzcd}
	0 & {S^{\,\oplus\,r_i}} & {M_i^{\,\prime}} & {M_i} & 0 \\
	0 & {S^{\,\oplus\,r_j}} & {M_j^{\,\prime}} & {M_j} & 0
	\arrow[from=1-1, to=1-2]
	\arrow["{f_i}", from=1-2, to=1-3]
	\arrow["\varphi"', dashed, from=1-2, to=2-2]
	\arrow["{g_i}", from=1-3, to=1-4]
	\arrow["\phi"', from=1-3, to=2-3]
	\arrow[from=1-4, to=1-5]
	\arrow["\psi", dashed, from=1-4, to=2-4]
	\arrow[from=2-1, to=2-2]
	\arrow["{f_j}"', from=2-2, to=2-3]
	\arrow["{g_j}"', from=2-3, to=2-4]
	\arrow[from=2-4, to=2-5]
\end{tikzcd}
\end{equation}
Then by the universal property of the kernel, there is a morphism $\varphi$ such that the left square in \eqref{mij} is commutative. By the snake lemma, we obtain $\psi$ is surjective, and $\operatorname{ker}\psi \cong \operatorname{coker}\varphi \in \operatorname{Gen}S=\operatorname{add}S$. Since $\operatorname{Hom}_A(S,M)=0$, we conclude that $\operatorname{ker}\psi=0$, and then $\psi$ is an isomorphism. This contradicts the fact that $M$ is basic. Hence, $M'$ is basic, and then
$|M' \oplus S|=|M\oplus S|=|W \oplus E|=|B|=|A|$.
Thus, $M'\oplus S$ is $\tau_A$-tilting.
Noting that $\operatorname{Hom}_A(S,M)=0$, i.e. $M\in\mathcal F(S)$, and $S^{\oplus r}\in\operatorname{Gen}S$, we conclude that
the exact sequence \eqref{tr} is the canonical decomposition of $M'$ with
respect to the torsion pair $(\operatorname{Gen}S,\mathcal F(S))$. By Lemma \ref{3.2}(1),
\(
\operatorname{Tr}_{M'}S\cong S^{\oplus r}\) and
 \(M'/\operatorname{Tr}_{M'}S\cong M.
\)
By \eqref{5=0}, we have
$\operatorname{Hom}_A\bigl(M'/\operatorname{Tr}_{M'}S,
 \tau_{A,2}(M'/\operatorname{Tr}_{M'}S)\bigr)=0.$
Moreover, since $\operatorname{Hom}_A(M',S)=0$ and
$\operatorname{Gen}S=\operatorname{add}S$, no nonzero
indecomposable direct summand of $M'$ lies in $\operatorname{Gen}S$.
Consequently,
$[M',S]\in\mathfrak D_{\tau_2}(A,T)$ and
\[
\Phi([M',S])
=\bigl[\operatorname{Hom}_A(T,M)\oplus\operatorname{Ext}_A^1(T,S)\bigr]
=[W\oplus E].
\]
Therefore, we complete the proof.
\end{proof}
\begin{corollary}\label{phi-bijection-bb}
\textup{(\cite[Proposition 3.3]{PENG})}
Let $A$ be a hereditary algebra, let
$T=\tau_A^{-1}S\oplus P$ be a BB-tilting module, and
$B=\operatorname{End}_A T$. Set
\[
\mathfrak D_1(A,T):=
\bigl\{[M,N]\ \bigm|\ M\in\mathcal T(T),\
N\in\mathcal F(T),\
M\oplus N\text{ is a basic }1\text{-tilting }A\text{-module}\bigr\}
\]
and
\[
\mathfrak T_2(B;T):=
\bigl\{[W\oplus E]\ \bigm|\
W\in\mathcal Y(T),\ E\in\mathcal X(T),\
W\oplus E\text{ is a basic }2\text{-tilting }B\text{-module}\bigr\}.
\]
Then the restriction map
$\Phi:\mathfrak D_1(A,T)\longrightarrow\mathfrak T_2(B;T)$
is a bijection.
\end{corollary}
\begin{proof}
Since $A$ is hereditary, the tilting triple $(A,T,B)$ is splitting.
By Lemma \ref{unique}, $\mathcal F(T)=\operatorname{add}S$. Take any
$[M,N]\in\mathfrak D_1(A,T)$. Note that each classical $1$-tilting module is
$\tau_A$-tilting. Since $\mathcal T(T)$ is closed under quotients,
$X:=M/\operatorname{Tr}_M N\in\mathcal T(T)$. Moreover, since
$\operatorname{Gen}N\subseteq\mathcal F(T)$ and every direct summand
of $M$ belongs to $\mathcal T(T)$, we conclude that no nonzero indecomposable
direct summand of $M$ lies in $\operatorname{Gen}N$. Finally, since
$A$ is hereditary, $\Omega_A X$ is projective and then
$\tau_{A,2}X=\tau_A\Omega_A X=0$. Hence, we have
$\mathfrak D_1(A,T)\subseteq\mathfrak D_{\tau_2}(A,T)$.
By Corollary \ref{3.4}, the  restriction map $\Phi$ is well-defined, and Proposition \ref{phi-injective} shows that it is injective.

It remains to prove it is surjective. Let
$[W\oplus E]\in\mathfrak T_2(B;T)$ and set $H:=W\oplus E$.
Then $\operatorname{Ext}_B^1(H,H)=0$ and $|H|=|B|$.
For any module $L\in\operatorname{Gen}H$, we have a short exact sequence
$0\to K\to H^{\oplus m}\to L\to 0$. Applying $\Hom_A(H,-)$ to it, we get $\operatorname{Ext}_B^2(H,L)=0$, since
$\operatorname{pd}_B H\leq2$ and
$\operatorname{Ext}_B^2(H,H)=0$. Thus, by Lemma \ref{tau-gen}, we obtain
$\operatorname{Hom}_B(H,\tau_{B,2}H)=0$. Hence, $H$ is $\tau_{B,2}$-tilting.
Moreover, by Theorem \ref{BB}, there exists $M_0\in\mathcal T(T)$ such that
$W\cong\operatorname{Hom}_A(T,M_0)$. By
\cite[VI, Lemma 4.1]{ASS},
$\operatorname{pd}_B W\leq\operatorname{pd}_A M_0\leq1$. Then by
$\operatorname{Ext}_B^1(W,W)=0$, we easily get
$\operatorname{Ext}_B^1(W,\operatorname{Gen}W)=0$. Thus, by Lemma \ref{tau-gen}, we have
$\operatorname{Hom}_B(W,\tau_BW)=0$. Using Proposition \ref{phi-surjective-bb}, we conclude that there exists
$[M,N]\in\mathfrak D_{\tau_2}(A,T)$ with
$M\in\mathcal T(T)$, $N\in\mathcal F(T)$ such that
$\Phi([M,N])=[W\oplus E]$. Since $A$ is hereditary,
each $\tau_A$-tilting module is a classical
$1$-tilting module. Hence, we have $[M,N]\in\mathfrak D_1(A,T)$, and complete the proof.
\end{proof}

\section{Comparisons with mutations of silting complexes}
In this section, we study the relationship between Corollary \ref{3.4} and mutations of silting complexes. Let $A$ be a finite-dimensional algebra. For any \( \boldsymbol{X} \in K^b(\operatorname{proj} A)\), denote by $\operatorname{thick} \boldsymbol{X}$ the smallest full subcategory of $K^b(\operatorname{proj} A)$ which contains $\boldsymbol{X}$ and is closed under positive and negative shifts, cones, isomorphisms and direct summands; denote by ${\rm H}^i\boldsymbol{X}$ the $i$-th cohomology of $\boldsymbol{X}$ for any $i\in\mathbb{Z}$. If $A$ has finite global dimension, it is well-known that $K^b(\operatorname{proj} A)$ is equivalent to $D^b(A)$ as triangulated categories.

\begin{definition}\label{silting}
Let \( \boldsymbol{X} = (X^i) \in K^b(\operatorname{proj} A) \).
The complex \( \boldsymbol{X} \) is called a {\em silting complex}, if
$\operatorname{Hom}_{K^b(\operatorname{proj} A)}(\boldsymbol{X}, \boldsymbol{X}[i]) = 0$ for any $i > 0$
and
$\operatorname{thick} \boldsymbol{X} = K^b(\operatorname{proj} A).$
The complex \( \boldsymbol{X} \) is called an {\em $n$-term complex}, if
$X^i = 0$ for any $i \notin \{0, \ldots, -(n-1)\}$
and ${\rm H}^{i} \boldsymbol{X} = 0$ for any $i \neq 0, -(n-1)$.
\end{definition}
Let us consider the following order: for any \( \boldsymbol{X}, \boldsymbol{Y} \in  K^b(\operatorname{proj} A) \), define
\begin{equation}\label{dydx}
 \boldsymbol{X} \geq \boldsymbol{Y}~~~~\text{if}~~~~
\Hom_{K^b(\operatorname{proj} A)}(\boldsymbol{X}, \boldsymbol{Y}[i]) = 0~\text{for~any}~ i > 0.\end{equation}
Using this order, we have the following (cf. \cite[Proposition 2.9]{MUC}).
\begin{lemma}\label{bjdx}
For any complex \( \boldsymbol{X} = (X^i) \in K^b(\operatorname{proj} A) \) and any integers $m\leq n$, the following statements are equivalent:
\begin{enumerate}
    \item[(1)] \( X^i = 0 \) for any \( i \notin [m, n] \);
    \item[(2)] \( A[-n] \geq \boldsymbol{X} \geq A[-m] \), where \( A \) is the stalk complex of the regular module \( A \).
\end{enumerate}
\end{lemma}

Let $\mathcal{A}$ be a category. A morphism $f:X \to Y$ in $\mathcal{A}$ is called {\em right minimal}, if every morphism $g:X\to X$ such that $fg=f$ is an automorphism.
We remark that a morphism $f:X \to Y$ in a Krull--Schmidt category is right minimal if and only if for any nonzero direct summand \( X' \) of \( X \), \( f|_{X'} \neq 0 \) (cf. \cite{BIA}). For a subcategory $\mathcal{X}$ of $\mathcal{A}$, the morphism $f:X \to Y$ is called a {\em right $\mathcal{X}$-approximation} of $Y$ if $X\in\mathcal{X}$ and $\Hom_{\mathcal{A}}(X',f)$ is surjective for any $X'\in \mathcal{X}$; it is called a {\em minimal right $\mathcal{X}$-approximation} of $Y$ if $f$ is both right minimal and a right $\mathcal{X}$-approximation of $Y$.

Now, let us recall the mutation of silting complexes (cf. \cite{AIHA,SOT}).
\begin{proposition}\label{mutation}
Let $\boldsymbol{X} \oplus \boldsymbol{Y}$ be a silting complex. Take a triangle
\[
\xymatrix{\boldsymbol{Z} \ar[r]& \boldsymbol{Y}' \ar[r]^-f& \boldsymbol{X} \ar[r]&\boldsymbol{Z}[1]}
\]
such that $f$ is a minimal right ${\rm add}\boldsymbol{Y}$-approximation of $\boldsymbol{X}$. Then $|\boldsymbol{Z} \oplus \boldsymbol{Y}|=|\boldsymbol{X} \oplus \boldsymbol{Y}|$ and $\boldsymbol{Z} \oplus \boldsymbol{Y}$ is a silting complex, called the {right mutation} of $\boldsymbol{X} \oplus \boldsymbol{Y} $ at $\boldsymbol{X}$.
\end{proposition}

Using Proposition \ref{mutation} and the derived equivalence in \eqref{RHom}, we have the following.
\begin{theorem}\label{main1}
Let $(A, T, B)$ be a tilting triple with $A$ hereditary. Let \( \boldsymbol{X} = (X^i,d_X^i), \boldsymbol{Y} = (Y^i,d_Y^i) \in K^b(\operatorname{proj} A) \) such that $d_X^{-1}$ and $d_Y^{-1}$ are right minimal.
Suppose $\boldsymbol{X} \oplus \boldsymbol{Y}$ is a basic $2$-term silting complex satisfying that
\({\rm H}^{0}\boldsymbol{X} \in \mathcal{T}(T)\) and \({\rm H}^{0} \boldsymbol{Y} \in \mathcal{F}(T).\)
Then for any triangle
\begin{equation}\label{star3}
\xymatrix{\boldsymbol{Z}[-1]\ar[r]&\boldsymbol{Y}^{\prime}\ar[r]^-f& \boldsymbol{X} \ar[r]&\boldsymbol{Z}}
\end{equation}
such that $f$ is a minimal right ${\rm add} \boldsymbol{Y}$-approximation of $\boldsymbol{X}$, the complex
\[
\mathbf{R}\operatorname{Hom}_A
   (T,\boldsymbol{Z}\oplus\boldsymbol{Y}[1])
\]
is isomorphic to a basic $3$-term silting complex $\boldsymbol{W}$ in
$K^b(\operatorname{proj}B)$.
\end{theorem}
\begin{proof}
By Proposition \ref{mutation}, $\boldsymbol{Z}[-1] \oplus \boldsymbol{Y}$ is a basic silting complex in $K^b(\operatorname{proj} A)$. Using the derived equivalence in \eqref{RHom}, we get that
$\mathbf{R}\operatorname{Hom}_A(T,\boldsymbol{Z}\oplus\boldsymbol{Y}[1])$
is isomorphic to a basic silting complex
$\boldsymbol{W}=(W^i,d_W^i)$ in $K^b(\operatorname{proj}B)$.

\noindent{\bf Claim 1:} $W^{i}=0$ for any $i \notin \{-2,-1,0\}$.
By Lemma \ref{bjdx}, in order to prove Claim $1$, it suffices to prove that $B\geq \boldsymbol{W}\geq B[2]$, where $B$ is the stalk complex of the regular module \( B \).
Note that
\begin{flalign*}
\bigoplus\limits_{i>0}\operatorname{Hom}_{K^b(\operatorname{proj}B)}(B,\boldsymbol{W}[i])
&\cong\bigoplus\limits_{i>0}\operatorname{Hom}_{D^b(B)}
\bigl(\mathbf{R}\operatorname{Hom}_A(T,T),
\mathbf{R}\operatorname{Hom}_A(T,\boldsymbol{Z}\oplus\boldsymbol{Y}[1])[i]\bigr)
\\
& \cong \bigoplus\limits_{i>0}\operatorname{Hom}_{D^b(A)}(T , (\boldsymbol{Z} \oplus \boldsymbol{Y}[1])\,[i])
\\
& \cong \bigoplus\limits_{i>0}\operatorname{Hom}_{D^b(A)}(T , \boldsymbol{Z} [i]) \bigoplus \bigoplus\limits_{i>0}\operatorname{Hom}_{D^b(A)}(T , \boldsymbol{Y}[i+1])
\\
& \cong \operatorname{Hom}_{D^b(A)}(T , \boldsymbol{Z}[1]) \bigoplus \bigoplus\limits_{i>0} \operatorname{Hom}_{D^b(A)}(\boldsymbol{P}_{T}, (\boldsymbol{Y} \oplus \boldsymbol{Z}) [i+1]),
\end{flalign*}
where $\boldsymbol{P}_{T}$ is the complex $P_{-1}\stackrel{\delta}{\rightarrow}P_0$ such that ${0\to P_{-1}\stackrel{\delta}{\rightarrow} P_0\to T\to0}$ is a minimal projective resolution  of $T$.

Since $\boldsymbol{X}$, $\boldsymbol{Y}$ are $2$-term complexes, set $\boldsymbol{Z}=(Z^i,d_Z^i)$, then it is easy to see $Z^{i}=0$ for any $i \notin \{-2,-1,0\}$.
Thus, it is direct to see
$$\operatorname{Hom}_{D^b(A)}(\boldsymbol{P}_{T}, (\boldsymbol{Y} \oplus \boldsymbol{Z}) [i+1])\cong\Hom_{K^b({\rm proj}A)}(\boldsymbol{P}_{T}, (\boldsymbol{Y} \oplus \boldsymbol{Z}) [i+1])=0~\text{for~any}~i>0.$$
Applying $\operatorname{Hom}_{D^b(A)}(T,-)$ to the triangle (\ref{star3}), we get the exact sequence
\[\xymatrix{\operatorname{Hom}_{D^b(A)}(T,\boldsymbol{X}[1]) \ar[r]& \operatorname{Hom}_{D^b(A)}(T,\boldsymbol{Z}[1]) \ar[r]& \operatorname{Hom}_{D^b(A)}(T,\boldsymbol{Y}'[2]).}
\]
Since $d_X^{-1}$ and $d_Y^{-1}$ are right minimal and $A$ is hereditary, by \cite[Proposition 5.5]{MAR}, we have
$\boldsymbol{X}\cong{\rm H}^{0} \boldsymbol{X}$ and $\boldsymbol{Y}\cong {\rm H}^{0} \boldsymbol{Y}$ in $D^b(A).$
Since ${\rm H}^{0}\boldsymbol{X} \in \mathcal{T}(T)$, we have $\operatorname{Hom}_{D^b(A)}(T,\boldsymbol{X}[1]) \cong \operatorname{Ext}_{A}^1(T,{\rm H}^{0}\boldsymbol{X})=0$. By ${\rm pd}T \leq 1$, we get $\operatorname{Hom}_{D^b(A)}(T,\boldsymbol{Y}'[2]) \cong \operatorname{Ext}_{A}^2(T,{\rm H}^{0}\boldsymbol{Y}')=0$. Thus, $\operatorname{Hom}_{D^b(A)}(T,\boldsymbol{Z}[1])=0$. Hence, we conclude that
$\operatorname{Hom}_{K^b(\text{proj}\,B)}(B , \boldsymbol{W}[i])=0$ for any $i>0$, i.e. $B\geq\boldsymbol{W}$.
On the other hand, for any $i>2$ we have
\begin{flalign*}
 \operatorname{Hom}_{K^b(\operatorname{proj}B)}(\boldsymbol{W},B[i])
 &\cong \operatorname{Hom}_{D^b(B)}
 \bigl(\mathbf{R}\operatorname{Hom}_A(T,\boldsymbol{Z}\oplus\boldsymbol{Y}[1]),
 \mathbf{R}\operatorname{Hom}_A(T,T)[i]\bigr) \\
& \cong \operatorname{Hom}_{D^b(A)}(\boldsymbol{Z} \oplus \boldsymbol{Y}[1],T[i])\\
&=0,
\end{flalign*}
where the last equality follows from the degrees of
$\boldsymbol{Z}\oplus\boldsymbol{Y}[1]$ and a projective resolution of $T$.
That is, $\boldsymbol{W}\geq B[2]$. This proves Claim 1.

\noindent{\bf Claim 2:} $\boldsymbol{W}$ is a $3$-term complex in $K^b({\rm proj}B)$.

By definition, we only need to prove ${\rm H}^{-1}\boldsymbol{W}=0$. Note that
\begin{equation}\label{ztongd}
{\rm H}^{-1}\boldsymbol{W}\cong{\rm H}^{-1}\mathbf{R}\Hom_A(T, \boldsymbol{Z} \oplus \boldsymbol{Y}[1])
\cong \operatorname{Hom}_{D^b(A)}(T, \boldsymbol{Z}[-1] \oplus \boldsymbol{Y}).\end{equation}
Since ${\rm H}^{0}\boldsymbol{Y} \in \mathcal{F}(T)$, we get $\operatorname{Hom}_{D^b(A)}(T, \boldsymbol{Y}) \cong \operatorname{Hom}_{A}(T, \operatorname{H}^{0} \boldsymbol{Y})=0$. Applying the functor $\operatorname{Hom}_{D^b(A)}(T,-)$ to the triangle (\ref{star3}), we get the exact sequence
\[\xymatrix{\Hom_{D^b(A)}(T,\boldsymbol{X}[-1]) \ar[r]& \Hom_{D^b(A)}(T,\boldsymbol{Z}[-1]) \ar[r]&  \Hom_{D^b(A)}(T,\boldsymbol{Y}').}\]
Since $\operatorname{Hom}_{D^b(A)}(T, \boldsymbol{Y})=0$, we get $\operatorname{Hom}_{D^b(A)}(T,\boldsymbol{Y}')=0$. Noting that $$\Hom_{D^b(A)}(T,\boldsymbol{X}[-1]) \cong \Hom_{D^b(A)}(T,({\rm H}^{0}\boldsymbol{X})[-1])=0,$$ we obtain $\operatorname{Hom}_{D^b(A)}(T,\boldsymbol{Z}[-1])=0$. Hence, by \eqref{ztongd}, Claim $2$ follows.

Therefore, we complete the proof.
\end{proof}

\begin{lemma}\label{add-gen}
Let $M, N \in {\rm mod}A$ and $f:M'\rightarrow N$ be a right ${\rm add}M$-approximation of $N$. Take the epi-monic factorization of $f$:
$$\xymatrix{M'\ar[rr]^-f\ar@{->>}[rd]_-{\tilde{f}}&&N\\
&\im f\ar@{^{(}->}[ru]_-{\sigma_f}&}$$
Then $\sigma_f$ is a right $\operatorname{Gen}M$-approximation of $N$.
\end{lemma}
\begin{proof}
Since $M'\in\operatorname{add}M$ and $\widetilde f$ is epic,
$\operatorname{Im}f$ belongs to $\operatorname{Gen}M$. For any morphism
$g:W\to N$ with $W\in\operatorname{Gen}M$, choose an epimorphism
$h:M''\twoheadrightarrow W$ with $M''\in\operatorname{add}M$. Since $f$ is a
right $\operatorname{add}M$-approximation, there is a morphism
$\alpha:M''\to M'$ such that $gh=f\alpha$. Thus,
$\operatorname{Im}(gh)\subseteq\operatorname{Im}f$. Since $h$ is epic,
$\operatorname{Im}(gh)=\operatorname{Im}g$, and then $g$ factors through
$\sigma_f:\operatorname{Im}f\hookrightarrow N$. Hence, $\sigma_f$ is a right
$\operatorname{Gen}M$-approximation.
\end{proof}

\begin{proposition}\label{mac}
Keep the same conditions and notation as in Theorem \ref{main1}. Suppose that
${\rm H}^{0}(\operatorname{cone}\lambda)\neq0$ for any nonzero morphism
$\lambda:\boldsymbol{Y}''\longrightarrow\boldsymbol{X}''$
in $K^b(\operatorname{proj}A)$,
where $\boldsymbol{Y}''\in{\rm add}\boldsymbol{Y}$ and
$\boldsymbol{X}''\in{\rm add}\boldsymbol{X}$, then
$\boldsymbol{Z}\cong{\rm H}^{0}\boldsymbol{Z}
\cong M/\operatorname{Tr}_MN$
in $D^b(A)$, where
$M={\rm H}^{0}\boldsymbol{X}$ and $N={\rm H}^{0}\boldsymbol{Y}$.
\end{proposition}

\begin{proof}
Since $A$ is hereditary, by
\cite[Proposition 2.1.2]{KAP} and the fact that $\boldsymbol{Z}$ is a $3$-term complex, we have the following isomorphism in $D^b(A)$
\[
\boldsymbol{Z}\cong{\rm H}^{0}\boldsymbol{Z} \oplus {\rm H}^{-1}\boldsymbol{Z}[1] \oplus {\rm H}^{-2}\boldsymbol{Z}[2].
\]
Applying the cohomology functor to the triangle \eqref{star3}, we get the exact sequence
\[\xymatrix{{\rm H}^{-2} \boldsymbol{X} \ar[r]& {\rm H}^{-2} \boldsymbol{Z} \ar[r]& {\rm H}^{-1} \boldsymbol{Y}'.}\]
Since $\boldsymbol{X}$ is a $2$-term complex, $\operatorname{H}^{-2} \boldsymbol{X}=0$.
Since $d_{Y'}^{\,-1}$ is right minimal, by \cite[Proposition 5.5]{MAR}, ${\rm H}^{-1}\boldsymbol{Y}'=0$. Thus ${\rm H}^{-2}\boldsymbol{Z}=0$.

Now it remains to prove ${\rm H}^{-1}\boldsymbol{Z}=0$. We may assume that
$\boldsymbol{Z}={\rm H}^{0}\boldsymbol{Z}\oplus
{\rm H}^{-1}\boldsymbol{Z}[1]$.

Since $d_X^{-1}$ and $d_Y^{-1}$ are right minimal, by \cite[Proposition 5.5]{MAR}, we have
 ${\rm H}^{-1}\boldsymbol{X}=0={\rm H}^{-1}\boldsymbol{Y}$. Let
$M={\rm H}^{0}\boldsymbol{X}\in\mathcal{T}(T)$ and
$N={\rm H}^{0}\boldsymbol{Y}\in\mathcal{F}(T)$. By
\cite[Theorem 3.1]{MUC} and \cite[Proposition 5.6]{MAR}, $M\oplus N$ is a
basic $\tau$-tilting module; since $A$ is hereditary, it is a basic tilting
module.

Applying the cohomology functor to \eqref{star3} and taking the epi--monic factorization of
${\rm H}^{0}f$, we obtain the following diagram whose row is exact:
\[\begin{tikzcd}
	0 & {{\rm H}^{-1}\boldsymbol{Z}} & {{\rm H}^{0}\boldsymbol{Y}'} & {{\rm H}^{0}\boldsymbol{X}=M} & {{\rm H}^{0}\boldsymbol{Z}} & 0 \\
	&&& {\im\,{\rm H}^{0}f}
	\arrow[from=1-1, to=1-2]
	\arrow[from=1-2, to=1-3]
	\arrow["{{\rm H}^{0}f}", from=1-3, to=1-4]
	\arrow[""', two heads, from=1-3, to=2-4]
	\arrow[from=1-4, to=1-5]
	\arrow[from=1-5, to=1-6]
	\arrow["{\sigma}"', hook, from=2-4, to=1-4]
\end{tikzcd}\]
Since $f$ is a minimal right ${\rm add}\boldsymbol{Y}$-approximation of
$\boldsymbol{X}$, by the dual of \cite[Lemma 6.9]{MAR}, we get that ${\rm H}^{0}f$
is a minimal right ${\rm add}N$-approximation of $M$. Set
$U=\operatorname{Im}{\rm H}^{0}f$. Lemma \ref{add-gen} shows that
$\sigma:U\hookrightarrow M$ is a right
$\operatorname{Gen}N$-approximation. We claim that
$U=\operatorname{Tr}_MN$. In fact, every morphism $N\to M$ factors through ${\rm H}^{0}f$, so its image is
contained in $U$ and $\operatorname{Tr}_MN\subseteq U$. Conversely, since
$U\in\operatorname{Gen}N$, take an epimorphism $N'\twoheadrightarrow U$ with
$N'\in\operatorname{add}N$, composing with $\sigma$ gives
$U\subseteq\operatorname{Tr}_MN$. Hence,
${\rm H}^{0}\boldsymbol{Z}\cong M/\operatorname{Tr}_MN$.

The cone condition implies that no nonzero indecomposable direct summand of
$M$ lies in $\operatorname{Gen}N$. Otherwise, there exists a direct summand $M_i$ of $ M$ such that
$M_i\in\operatorname{Gen}N$, choose an epimorphism $p:N'\twoheadrightarrow M_i$
with $N'\in\operatorname{add}N$. Under the identifications
$\boldsymbol{X}\cong M$ and $\boldsymbol{Y}\cong N$, the map $p$ is
represented by a nonzero morphism
$\lambda:\boldsymbol{Y}''\to\boldsymbol{X}_i$ in
$K^b(\operatorname{proj}A)$, with
$\boldsymbol{Y}''\in\operatorname{add}\boldsymbol{Y}$ and
$\boldsymbol{X}_i\in\operatorname{add}\boldsymbol{X}$. The cohomology exact
sequence gives ${\rm H}^{0}(\operatorname{cone}\lambda)=\operatorname{Coker}p=0$,
contrary to the hypothesis.

Let $Q=M\oplus N$. Since $U\in\operatorname{Gen}N\subseteq\operatorname{Gen}Q$ and $Q$ is
$\tau$-rigid, by Lemma \ref{tau-gen}, we get
$\operatorname{Ext}_A^1(Q,U)=0$. Applying $\operatorname{Hom}_A(-,Q)$ and
$\operatorname{Hom}_A(-,U)$ to
\[
0\longrightarrow U\longrightarrow M\longrightarrow M/U\longrightarrow0,
\]
we get $\operatorname{Ext}_A^1(U,Q)=0=\operatorname{Ext}_A^1(U,U)$. Thus, $Q\oplus U$
is partial tilting. Since the basic tilting module $Q$ has already $|A|$
indecomposable summands, we conclude that
$U\in\operatorname{add}Q$.
Note that $U\in\operatorname{Gen}N$ and $\operatorname{Gen}N$ is closed under direct summands.
Then, since no nonzero indecomposable direct summand of
$M$ lies in $\operatorname{Gen}N$, we have $U\in\operatorname{add}N$. Thus, the monomorphism
$\sigma:U\hookrightarrow M$ is a minimal right
$\operatorname{add}N$-approximation. By the uniqueness of minimal right
approximations, it is isomorphic to ${\rm H}^{0}f$. Hence, we obtain that the epimorphism
${\rm H}^{0}\boldsymbol{Y}'\twoheadrightarrow U$ is an isomorphism, and then ${\rm H}^{-1}\boldsymbol{Z}=0$.

Therefore, we complete the proof.
\end{proof}

In the following, let us illustrate that Corollary \ref{3.4} can be implied by Theorem \ref{main1}.
\begin{corollary}\label{comparison}
Keep the same conditions and notation as given in Theorem \ref{main1}. Suppose
that ${\rm H}^{0}(\operatorname{cone}\lambda)\neq0$ for every nonzero morphism
$\lambda:\boldsymbol{Y}''\to\boldsymbol{X}''$ in
$K^b(\operatorname{proj}A)$, where
$\boldsymbol{Y}''\in{\rm add}\boldsymbol{Y}$ and
$\boldsymbol{X}''\in{\rm add}\boldsymbol{X}$. Then
${\rm H}^0\boldsymbol{W}$ is a basic $2$-tilting $B$-module and is isomorphic
to $\operatorname{Hom}_{A}(T,M/\operatorname{Tr}_{M}N)
\oplus\operatorname{Ext}_{A}^{1}(T,N),$
where $M={\rm H}^{0}\boldsymbol{X}$ and
$N={\rm H}^{0}\boldsymbol{Y}$.
\end{corollary}
\begin{proof}
By Proposition \ref{mac}, we have
$\boldsymbol{Z}\cong{\rm H}^{0}\boldsymbol{Z}
\cong M/\operatorname{Tr}_MN$. Using
\cite[Proposition 5.5]{MAR}, we get $\boldsymbol{Y}\cong N$. Since
$M/\operatorname{Tr}_MN\in\mathcal{T}(T)$ and
$N\in\mathcal{F}(T)$, by the derived Brenner--Butler correspondence, we have
\begin{align*}
\boldsymbol{W}
\cong\mathbf{R}\operatorname{Hom}_A
   (T,\boldsymbol{Z}\oplus\boldsymbol{Y}[1])\cong
\operatorname{Hom}_A(T,M/\operatorname{Tr}_MN)
\oplus\operatorname{Ext}_A^1(T,N)=:H,
\end{align*}
where $H$ is a stalk complex in degree zero. In particular,
${\rm H}^{0}\boldsymbol{W}\cong H$.
Since $\boldsymbol{W}$ is a three-term complex of projective $B$-modules in degrees
$-2,-1,0$ and is isomorphic to $H$, we conclude that it is a projective resolution of $H$
and thus $\operatorname{pd}_B H\leq2$. Since $\boldsymbol{W}$ is silting, we get
$\operatorname{Ext}_B^i(H,H)
\cong\operatorname{Hom}_{D^b(B)}(\boldsymbol{W},\boldsymbol{W}[i])=0$
for any $i>0$ and
$\operatorname{thick}H=\operatorname{thick}\boldsymbol{W}
=K^b(\operatorname{proj}B)$. By the standard derived characterization of
classical tilting modules, $H$ is a $2$-tilting module. Finally, since
$\boldsymbol{W}$ is basic, we conclude that so is $H$.
\end{proof}


\section{Triangular endomorphism algebras and derived decompositions}
In this section, we describe the endomorphism algebra of the \( \tau_{2} \)-tilting module constructed by using Theorem \ref{3.3} and give the corresponding
derived decomposition when this module is $2$-tilting.
For related constructions involving triangular matrix algebras, we refer to
\cite{LAD,LI}.

Let $(A,T,B)$ be a tilting triple with $X\in\mathcal{T}(T)$ and
$N\in\mathcal{F}(T)$. Set
\[
H=\operatorname{Hom}_A(T,X)\oplus\operatorname{Ext}_A^1(T,N),
\quad L=X\oplus N[1],
\]
and write
\[
R=\operatorname{End}_A(X),\quad S=\operatorname{End}_A(N),\quad
U=\operatorname{Ext}_A^1(X,N).
\]
The $(S,R)$-bimodule structure on $U$ is given by
$\beta\cdot\xi\cdot\alpha=\beta_*(\alpha^*\xi)$, where $\alpha^*$ and
$\beta_*$ denote the pullback and pushout of extensions, respectively.
The derived equivalence
$F=\mathbf{R}\operatorname{Hom}_A(T,-):D^b(A)\longrightarrow D^b(B)$
from \eqref{RHom} satisfies
\[
F(X)\cong\operatorname{Hom}_A(T,X),\quad
F(N[1])\cong\operatorname{Ext}_A^1(T,N).
\]
Thus $F(L)\cong H$ in $D^b(B)$. Since
$\operatorname{Hom}_{D^b(A)}(N[1],X)=0$ and
$\operatorname{Hom}_{D^b(A)}(X,N[1])\cong U$, we have
\begin{equation}\label{triangular-algebra}
C:=\operatorname{End}_B(H)
\cong\operatorname{End}_{D^b(B)}(F(L))
\cong\operatorname{End}_{D^b(A)}(L)
\cong\begin{pmatrix}R&0\\U&S\end{pmatrix},
\end{equation}
where the last matrix denotes the triangular matrix algebra corresponding to the $(S,R)$-bimodule $U$.

Let $\mathcal{D}$ be a triangulated category. According to \cite[Definition~2.3]{KUZ} and \cite[Section~4]{BOK}, we recall that a full semi-orthogonal decomposition of
$\mathcal D$, denoted by $\mathcal D=\langle\mathcal A,\mathcal B\rangle$, consists of two full
triangulated subcategories $\mathcal A,\mathcal B$ satisfying
$\operatorname{Hom}_{\mathcal D}(\mathcal B,\mathcal A[i])=0$ for all
$i\in\mathbb Z$, such that every $D\in\mathcal{D}$ fits into a triangle with $A\in\mathcal A$ and $B\in\mathcal B$:
\[
B\longrightarrow D\longrightarrow A\longrightarrow B[1].
\]

\begin{theorem}\label{derived-triangular-application}
With the notation above, suppose that $H$ is a $2$-tilting $B$-module.
Then the following hold.
\begin{enumerate}
\item[(1)] $L$ is isomorphic to a tilting complex in
$K^b(\operatorname{proj}A)$, and there are derived equivalences
\begin{equation}\label{lgdcdj}
D^b(A)\mathop{\longrightarrow}^{F}D^b(B)
\mathop{\longrightarrow}^{\mathbf{R}\operatorname{Hom}_B(H,-)}D^b(C)\end{equation}
such that the composition of equivalences is the tilting equivalence associated to $L$.
\item[(2)] There is a full semi-orthogonal decomposition
\begin{equation}\label{semiorthogonal-pair}
K^b(\operatorname{proj}A)
=\left\langle\operatorname{thick}X,\operatorname{thick}N\right\rangle
\end{equation}
such that $\operatorname{thick}X$ and $\operatorname{thick}N$ are triangle equivalent to
$K^b(\operatorname{proj}R)$ and $K^b(\operatorname{proj}S)$,
respectively.
\end{enumerate}
\end{theorem}

\begin{proof}
(1) Since $H$ is a $2$-tilting $B$-module, it is easy to see that each projective resolution of $H$
is a tilting complex.
Since $F$ restricts to an triangle equivalence between $K^b(\operatorname{proj}A)$ and $K^b(\operatorname{proj}B)$, by $F(L)\cong H$, we get $L$ is isomorphic to a tilting complex in
$K^b(\operatorname{proj}A)$, in particular, $X,N\in K^b(\operatorname{proj}A)$. Rickard's theorem \cite[Theorem 6.4]{RIC} gives the derived
equivalences in \eqref{lgdcdj}. By $F(L)\cong H$ and \eqref{triangular-algebra}, we conclude that the tilting equivalence associated to $L$ is the same as the composition of equivalences in \eqref{lgdcdj}.

(2) We first prove that there is a full semi-orthogonal decomposition for $K^b(\operatorname{proj}C)$. Let $e$ and $f=1-e$
be the diagonal idempotents corresponding to $R$ and $S$ in
\eqref{triangular-algebra}. Under the composition
equivalence in \eqref{lgdcdj}, $X$ and $N[1]$
correspond to $eC$ and $fC$, respectively. Moreover,
\(
\operatorname{Hom}_C(fC,eC)=0.
\)
For any bounded complex $P$ of finitely generated projective $C$-modules,
write each component as $P^i=P_e^i\oplus P_f^i$ with
$P_e^i\in\operatorname{add}(eC)$ and $P_f^i\in\operatorname{add}(fC)$.
Then we get a subcomplex $P_f$ of $P$  and  a
degreewise split exact sequence of complexes
\[
0\longrightarrow P_f\longrightarrow P\longrightarrow P_e
\longrightarrow0,
\]
which gives a triangle in $K^b(\operatorname{proj}A)$.  Since $eC$ and $fC$ are projective, it is easy to see
$\operatorname{Hom}_{K^b(\operatorname{proj}C)}
(\operatorname{thick}(fC),\operatorname{thick}(eC)[i])=0$ for all $i\in\mathbb{Z}$.
Thus, we obtain the full semi-orthogonal decomposition
\[
K^b(\operatorname{proj}C)
=\left\langle\operatorname{thick}(eC),\operatorname{thick}(fC)
\right\rangle.
\]
Noting that we have the equivalences
$\operatorname{add}(eC)\simeq\operatorname{proj}R$ and
$\operatorname{add}(fC)\simeq\operatorname{proj}S$, we get the triangle equivalences $\operatorname{thick}(eC)\simeq K^b(\operatorname{proj}R)$ and
$\operatorname{thick}(fC)\simeq K^b(\operatorname{proj}S)$. Using the composition
equivalence in \eqref{lgdcdj} and $\operatorname{thick}(N[1])=\operatorname{thick}N$, we get the full semi-orthogonal decomposition \eqref{semiorthogonal-pair}.
\end{proof}

\begin{corollary}\label{hereditary-triangular-application}
Under the hypotheses of Corollary \ref{3.4}, take
$X=M/\operatorname{Tr}_MN$ and retain the notation above. Then we have a full semi-orthogonal decomposition
\[
D^b(A)=\left\langle\operatorname{thick}X,
\operatorname{thick}N\right\rangle
\]
such that $\operatorname{thick}X$ and $\operatorname{thick}N$ are triangle equivalent to $D^b(R)$ and $D^b(S)$,
respectively.
\end{corollary}

\begin{proof}
By Corollary \ref{3.4},   $H$ is
$2$-tilting.
Since $A$ is hereditary, $K^b(\operatorname{proj}A)=D^b(A)$. Noting that
\[
\max\{\operatorname{gl.dim}R,\operatorname{gl.dim}S\}
\leq\operatorname{gl.dim}C<\infty,
\]
we have $K^b(\operatorname{proj}R)=D^b(R)$ and
$K^b(\operatorname{proj}S)=D^b(S)$. Hence, using Theorem \ref{derived-triangular-application}, we finish the proof.
\end{proof}

Let us provide the following example with $A$ non-hereditary to illustrate Theorem \ref{derived-triangular-application}.

\begin{example}\label{nonhereditary-triangular-example}
Retain the notation of Example \ref{nonhereditary-splitting-example}.
Thus $X=I_2\oplus S_2\oplus P_2\oplus P_5$, $N=S_3$, and
$H\cong I_{3'}\oplus P_{2'}\oplus P_{3'}\oplus P_{5'}\oplus S_{2'}$
is a basic $2$-tilting $B$-module. The endomorphism algebra of $H$ is
\[
C:=\operatorname{End}_B(H)
\cong kQ_C/\langle dc_1c_2,c_2c_3\rangle,
\qquad
Q_C:\quad
0\overset d\longrightarrow1
\overset{c_1}\longrightarrow2
\overset{c_2}\longrightarrow3
\overset{c_3}\longrightarrow4.
\]
Moreover, we have
\[
R=\operatorname{End}_A(X)
\cong k(1\xrightarrow{c_1}2\xrightarrow{c_2}3\xrightarrow{c_3}4)
/\langle c_2c_3\rangle,
\qquad S=\operatorname{End}_A(N)\cong k.
\]
By Theorem \ref{derived-triangular-application},
$A$, $B$ and $C$ are pairwise derived equivalent, and there is a full semi-orthogonal decomposition:
\[
K^b(\operatorname{proj}A)=\left\langle \operatorname{thick}X,
\operatorname{thick}N\right\rangle
=\left\langle K^b(\operatorname{proj}R),
K^b(\operatorname{proj}S)\right\rangle,
\]
see Figure \ref{fig:example53}.
\begin{figure}[htbp]
\centering
\[
K^b(\operatorname{proj}A)=
\left\langle
\textcolor{sodA}{K^b(\operatorname{proj}R)},
\textcolor{sodB}{K^b(\operatorname{proj}S)}
\right\rangle
\]
\begin{tikzpicture}[x=.8cm,y=.5cm,
  sodarr/.style={-{Stealth[length=1.1mm]},draw=black!60,line width=.35pt},
  sodcont/.style={sodarr,dashed,draw=black!35},
  soddot/.style={circle,inner sep=0pt,minimum size=2.2mm,
    outer sep=.5pt,draw=none}]
\foreach \m/\indices in {0/{1,2,3,4},1/{0,4},2/{1},3/{0,3,4},
                         4/{0,2,3,4},5/{1,4},6/{0},7/{1,3,4}}{
  \foreach \i/\parity/\yy in {0/1/3,1/1/1,2/0/2,3/1/-1,4/0/-2}{
    \node[soddot,fill=sodM] (sod\m-\i) at ({2*\m+\parity},\yy) {};
  }
  \foreach \i in \indices {
    \node[soddot,fill=sodA] at (sod\m-\i) {};
  }
}
\node[soddot,fill=sodB] at (sod0-0) {};
\node[soddot,fill=sodB] at (sod4-1) {};
\foreach \m in {0,...,7}{
  \foreach \j in {0,1,3}{\draw[sodarr] (sod\m-2)--(sod\m-\j);}
  \draw[sodarr] (sod\m-4)--(sod\m-3);
}
\foreach \m in {0,...,6}{
  \pgfmathtruncatemacro{\nextm}{\m+1}
  \foreach \j in {0,1,3}{\draw[sodarr] (sod\m-\j)--(sod\nextm-2);}
  \draw[sodarr] (sod\m-3)--(sod\nextm-4);
}
\foreach \yy in {3,1,-1}{\draw[sodcont] (-.7,\yy)--(sod0-2);}
\draw[sodcont] (-.7,-1)--(sod0-4);
\foreach \j in {0,1,3}{\draw[sodcont] (sod7-\j)--(15.75,2);}
\draw[sodcont] (sod7-3)--(15.75,-2);
\end{tikzpicture}
\caption{The semi-orthogonal decomposition of $K^b(\operatorname{proj}A)$. Blue points represent $K^b(\operatorname{proj}R)$ and orange points represent $K^b(\operatorname{proj}S)$. }
\label{fig:example53}
\end{figure}
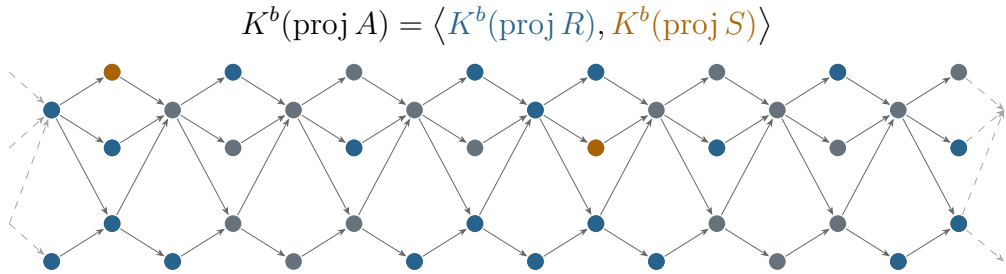
\end{example}

\section*{Acknowledgments}
H. Zhang was partially supported by the National Natural Science Foundation of China (No.~12271257) and the Natural Science Foundation of Jiangsu Province of China (No.~BK20240137). T. Zhao was partially supported by the National Natural Science Foundation of China (No.~12471036) and the Hubei Provincial Natural Science Foundation of China (No.~2026AFA094).

\end{document}